\documentclass[12pt]{article}

\usepackage{amsmath, amssymb, amsthm, physics, url, mathrsfs, cite, xcolor, mathtools, dsfont, enumitem, esint}
\mathtoolsset{showonlyrefs}
\usepackage[utf8]{inputenc}
\allowdisplaybreaks[4]

\definecolor{darkgreen}{rgb}{0,0.75,0}
\definecolor{darkred}{rgb}{0.75,0,0}
\definecolor{darkmagenta}{rgb}{0.5,0,0.5}
\usepackage[colorlinks,
	citecolor=darkgreen, 
	linkcolor=darkred, 
	urlcolor=darkmagenta,
	bookmarks=false,
	pdfauthor={Aobo Chen and Mathav Murugan},
	pdftitle={Bounding martingale dimension by spectral dimension}
]{hyperref}

\numberwithin{equation}{section}
\numberwithin{figure}{section}

\newtheorem{theorem}{Theorem}[section]
\newtheorem{lemma}[theorem]{Lemma}
\newtheorem{proposition}[theorem]{Proposition}

\theoremstyle{definition}
\newtheorem{definition}[theorem]{Definition}

\newtheorem{notation}[theorem]{Notation}

\DeclareMathOperator*{\esssup}{\operatorname*{ess\,sup}}

\def\dist{{\mathop{{\rm dist}}}}
\def\diam{{\mathop{{\rm diam }}}}
\newcommand*{\dif}{\mathop{}\!\mathrm{d}}
\newcommand{\supp}{\operatorname{supp}}
\newcommand{\one}{\mathds{1}}

\def\sG {{\mathcal G}}  
  \def\sL {{\mathcal L}}
\def\sM {{\mathcal M}}

 \def\bN {{\mathbb N}} 
  \def\bR {{\mathbb R}}

 \def\bZ {{\mathbb Z}}

\newcommand{\loc}[0]{\operatorname{loc}}
\newcommand{\Sett}[2]{\left\{ #1  : \, #2 \right\}}

\newcommand\restr[2]{{
		\left.\kern-\nulldelimiterspace 
		#1 
		\vphantom{\big|}
		\right|_{#2}
}}

\newcommand{\ambient}{X}
\newcommand{\metric}{d}
\newcommand{\meas}{m}
\DeclareMathOperator*{\osc}{osc}

\newcommand{\form}{\mathcal{E}} 
\newcommand{\domain}{\mathcal{F}}
\newcommand{\gen}{\Delta}
\newcommand{\Dom}{\mathrm{Dom}}
\newcommand{\MD}{\hyperlink{MD}{\mathrm{MD}}}
\newcommand{\VD}{\hyperlink{VD}{\mathrm{VD}}}
\newcommand{\PI}[1]{\hyperlink{PI}{\mathrm{PI}( #1 )}}
\newcommand{\Capa}[1]{\hyperlink{Cap}{\mathrm{Cap}_{\le}( #1 )}}
\newcommand{\Capacity}{\mathrm{Cap}}

\font\titlefont=cmbx14 scaled\magstep1
\title{\titlefont Bounding martingale dimension by spectral dimension}

\date{\today}
\author{
Aobo Chen\footnote{Research partially supported by the research grant (VIL73729) from Villum Fonden.}
\ and 
Mathav Murugan\footnote{Research partially supported by NSERC   discovery grant (RGPIN-2024-06125) and the Canada research chairs program (CRC-2023-00048)}
}

\begin{document}

\maketitle
\vspace{-0.5cm}
\begin{abstract}
We prove that the martingale dimension of a broad class of strongly local regular Dirichlet forms is bounded above by	a small-scale version of spectral dimension. This resolves a conjecture of Hino for this class of spaces, extending his earlier results beyond self-similar fractals. Our proof uses differentiation of energy measures to select good blow-up scales and construct approximately harmonic coordinates. We then study weak limits of their Euclidean image measures and combine an integration-by-parts argument with uniform Sobolev inequalities to establish the dimension bound.


\end{abstract}

\section{Introduction}\label{s.intro}

The study of diffusion processes on fractals has revealed phenomena that are absent in classical Brownian motion on Euclidean spaces. Following the pioneering work of Barlow and
Perkins on the Sierpi\'nski gasket \cite{BP88}, sub-Gaussian heat kernel estimates have been established for many fractal diffusions, including those on nested and affine nested
fractals \cite{Kum93,FHK94}, Sierpi\'nski carpets \cite{BB99}, and certain random Sierpi\'nski carpets \cite{HKKZ00}. See also \cite{HK99,Bar98} for further results on fractal diffusions. More flexible constructions with prescribed
volume growth and escape-time exponents were obtained by Barlow \cite{Bar04} for random walks on graphs, and subsequently by the second-named author \cite{Mur25}
for diffusions and random walks with general volume growth and escape-time profiles. They often exhibit anomalous diffusion, with space-time scaling different from
that of Euclidean Brownian motion. One manifestation of this phenomenon is  the appearance of sub-Gaussian heat kernel
estimates.

To describe this phenomenon, consider first a
space whose reference measure $m$ has
homogeneous volume growth
\[
m(B(x,r))\asymp r^\alpha,
\]
where $\alpha>0$ is called the
\emph{volume growth exponent}.
For many natural diffusions on such spaces,
the expected time needed to exit a ball of
radius $r$, when started at its center, is
comparable to $r^\beta$.
The exponent $\beta$, called the
\emph{walk dimension}, describes the
characteristic space-time scaling of the
diffusion. For Brownian motion on Euclidean
space, $\beta=2$, whereas on many fractals
one has $\beta>2$, reflecting the slower
spread of the diffusion.

This behavior is reflected in the
sub-Gaussian heat kernel estimates, in which
the heat kernel $p_t(x,y)$ is bounded above
and below by expressions of the form
\[
\frac{1}{t^{\alpha/\beta}}
\exp\left(
-c\left(\frac{d(x,y)^\beta}{t}
\right)^{1/(\beta-1)}
\right),
\]
with possibly different positive constants
in the upper and lower bounds.
The case $\beta=2$ recovers the familiar
Gaussian estimates. In particular, the
on-diagonal behavior is governed by
\[
p_t(x,x)\asymp t^{-\alpha/\beta}.
\]

The quantity
\[
d_{\mathrm s}=\frac{2\alpha}{\beta}
\]
is called the \emph{spectral dimension}. The terminology comes from the spectral behavior of the generator of the diffusion;
see \cite[Section~4.1]{Kig01} and \cite{BK97}.
For example, on a compact space with the above heat kernel estimates, if $N(\lambda)$ denotes the number of
eigenvalues of the nonnegative Laplacian not exceeding $\lambda$, counted with multiplicity, then
\[
N(\lambda)\asymp\lambda^{d_{\mathrm s}/2}
\qquad\text{as }\lambda\to\infty.
\]
Thus the spectral dimension plays the same role in eigenvalue growth that the Euclidean dimension plays in Weyl's law.
Equivalently, it governs the short-time growth of the heat trace.

For Brownian motion on $\mathbb R^n$, the volume growth exponent, walk dimension, and spectral dimension are respectively $n$, $2$, and $n$. For the standard Sierpi\'nski gasket \cite{BP88},  these exponents are
\[
\alpha=\frac{\log3}{\log2},\qquad
\beta=\frac{\log5}{\log2},\qquad
d_{\mathrm s}=\frac{2\log3}{\log5}.
\]
More generally, these exponents provide different descriptions of the geometry and diffusion behavior of the underlying
space. The volume growth exponent describes the size of metric balls, the walk dimension describes how quickly the diffusion moves, and the spectral dimension combines these
two aspects.

A further notion of dimension, which is the main subject of this paper, is the \emph{martingale dimension}. Informally, it measures the number of independent sources of martingale noise needed to describe a diffusion process. More precisely, it is the smallest number of martingale additive functionals with  respect to which every finite-energy martingale additive functional can be represented by stochastic integration. For Brownian motion on $\mathbb R^n$, this number is $n$, corresponding to the $n$ coordinate Brownian motions. In contrast, the natural Brownian motion on the Sierpi\'nski gasket has martingale dimension one, even though the gasket has Hausdorff dimension greater than one \cite{Kus89,Hin10}. Thus martingale dimension captures a different aspect of diffusion that is not directly determined by the Hausdorff dimension of the underlying space or the walk dimension. A fundamental question is how martingale dimension is constrained by the heat kernel behavior of the diffusion. Hino \cite{Hin13} proved that martingale dimension is bounded above by spectral dimension for certain diffusions on self-similar fractals and conjectured that such a bound should hold in much greater generality. The main result of this paper establishes this bound for a broad class of metric measure Dirichlet  spaces without assuming self-similarity.

Our result also addresses one direction of a more general problem concerning the possible combinations of these dimensions. In \cite[Conjecture~5.18]{Mur25}, the second-named author conjectured that for any   $2\le\beta\le\alpha+1$,
and integer $d_{\mathrm m}$ satisfying
\[
1\le d_{\mathrm m}\le\frac{2\alpha}{\beta},
\]
there exists a symmetric diffusion on a geodesic metric measure space satisfying two-sided sub-Gaussian heat kernel estimates with volume growth exponent $\alpha$, walk dimension $\beta$, and martingale dimension $d_{\mathrm m}$. The present work establishes the necessity of the upper bound
$d_{\mathrm m}\le2\alpha/\beta$. The complementary problem of constructing diffusions realizing every admissible triple $(\alpha,\beta,d_{\mathrm m})$ remains open in general.

We now introduce the analytic framework and formulate our main result precisely.

\subsection{Dirichlet forms and their energy measures}

We recall the basic definitions and notation concerning symmetric Dirichlet forms and their energy measures. We refer to \cite{FOT11} for a comprehensive treatment of Dirichlet forms.

Let $(\ambient,\metric)$ be a proper separable metric space containing at least two points, and let $\meas$ be a Radon measure on $\ambient$ with full support. Thus $\meas$ is finite on compact sets and $\meas(U)>0$ for every nonempty open set $U\subset \ambient$. We call $(\ambient,\metric,\meas)$ a \emph{metric measure space}. For $x\in \ambient$ and $r>0$, write
\[B(x,r)=\{y\in \ambient:d(x,y)<r\}, \qquad V(x,r)=\meas(B(x,r)).\]
We also write $\diam(E,\metric)=\sup_{x,y\in E}d(x,y)$ for any subset $E\subset \ambient$.

Furthermore let $(\form,\domain)$ be a \emph{symmetric Dirichlet form} on $L^{2}(\ambient,\meas)$; by definition, $\domain$ is a dense linear subspace of $L^{2}(\ambient,\meas)$, and $\form:\domain\times\domain\to\mathbb{R}$ is a non-negative definite symmetric bilinear form which is \emph{closed} ($\domain$ is a Hilbert space under the inner product $\form_{1}:= \form+ \langle \cdot,\cdot \rangle_{L^{2}(\ambient,\meas)}$) and \emph{Markovian} ($f^{+}\wedge 1\in\domain$ and $\form(f^{+}\wedge 1,f^{+}\wedge 1)\leq \form(f,f)$ for any $f\in\domain$). Recall that $(\form,\domain)$ is called \emph{regular} if $\domain\cap\mathcal{C}_{\mathrm{c}}(\ambient)$ is dense both in $(\domain,\form_{1})$ and in $(\mathcal{C}_{\mathrm{c}}(\ambient),\|\cdot\|_{\mathrm{sup}})$, and that $(\form,\domain)$ is called \emph{strongly local} if $\form(f,g)=0$ for any $f,g\in\domain$ with $\supp_{m}[f]$, $\supp_{m}[g]$ compact and $\supp_{m}[f-a\one_{\ambient}]\cap\supp_{m}[g]=\emptyset$ for some $a\in\mathbb{R}$. Here for a Borel measurable function $f:\ambient\to[-\infty,\infty]$ or an $\meas$-equivalence class $f$ of such functions, $\supp_{m}[f]$ denotes the support of the measure $|f|\dif \meas$, i.e., the smallest closed subset $F$ of $\ambient$ with $\int_{\ambient\setminus F}|f|\dif \meas=0$, which exists since $\ambient$ has a countable open base for its topology; note that $\supp_{m}[f]$ coincides with the closure of $\ambient\setminus f^{-1}(\{0\})$ in $\ambient$ if $f$ is continuous. The pair $(\ambient,\metric,\meas,\form,\domain)$ of a metric measure space $(\ambient,\metric,\meas)$ and a strongly local, regular symmetric Dirichlet form $(\form,\domain)$ on $L^{2}(\ambient,\meas)$ is termed a \emph{metric measure Dirichlet space}, or a \emph{MMD space} in abbreviation. We refer to \cite{FOT11, BH91} for details of the theory of symmetric Dirichlet forms.

Associated with a symmetric Dirichlet form is a \emph{strongly continuous contraction semigroup} $(P_t)_{t > 0}$; that is, a family of symmetric bounded linear operators $P_t:L^2(\ambient,\meas) \to L^2(\ambient,\meas)$ such that
  \begin{equation*}
  	P_{t+s}f=P_t(P_s f), \quad \norm{P_t f}_2 \le \norm{f}_{L^{2}(\ambient,\meas)}, \quad \lim_{t \downarrow 0} \norm{P_t f-f}_{L^{2}(\ambient,\meas)} =0, 
  \end{equation*}
  for all $t,s\in(0,\infty)$ and all $f \in L^2(\ambient,\meas)$. In this case, as stated in \cite[Lemma 1.3.4-(i)]{FOT11} we can express $(\form,\domain)$ in terms of the semigroup as
  \begin{equation} \label{e:semigroup}
  	\begin{split}
  		\domain &= \biggl\{f \in L^{2}(\ambient,\meas) \biggm| \lim_{t \downarrow 0} \frac{1}{t}\langle f - P_{t} f, f \rangle_{L^{2}(\ambient,\meas)} < \infty \biggr\}, \\
  		\form&(f,f) = \lim_{t \downarrow 0} \frac{1}{t}\langle f - P_{t} f, f \rangle_{L^{2}(\ambient,\meas)} \qquad \textrm{for all $f \in \domain$.}
  	\end{split}
  \end{equation}
The associated nonpositive self-adjoint generator $(\gen,\Dom(\gen))$ is defined by
\[\Dom(\gen):=\Sett{u\in L^2(\ambient,\meas)}{\lim_{t\downarrow0}\frac{P_tu-u}{t}\text{ exists in }L^2(\ambient,\meas)},\]
and
\[\gen u:=\lim_{t\downarrow0}\frac{P_tu-u}{t},\qquad u\in\Dom(\gen).\]

This nonpositive self-adjoint generator $(\gen,\Dom(\gen))$ is characterized by
\[\Dom(\gen)=
\Sett{u\in\domain}{\begin{array}{l}
	\text{there exists }f\in L^2(\ambient,\meas)\text{ such that}\\
	\form(u,v)=-\langle f,v\rangle_{L^2(\ambient,\meas)}
	\text{ for every }v\in\domain
\end{array}}.\]
For $u\in\Dom(\gen)$, the function
$f$ is uniquely determined, and we set $\gen u=f$.
Equivalently,
\begin{equation}\label{e.generator-def}
	\form(u,v)
	=-\langle\gen u,v\rangle_{L^2(\ambient,\meas)},
	\ \text{ for all }\ v\in\domain.
\end{equation}
See \cite[Section~1.3]{FOT11}.

For $\alpha\in(0,\infty)$, the \emph{resolvent} associated with the semigroup is given by
\[R_\alpha=(\alpha-\gen)^{-1}=\int_0^\infty e^{-\alpha t}P_t\dif t,\]
where the integral converges strongly on $L^2(\ambient,\meas)$. Moreover, $R_\alpha$ maps $L^2(\ambient,\meas)$ onto $\Dom(\gen)$ and satisfies $(\alpha-\gen)R_\alpha f=f$ \cite[(1.3.1)--(1.3.3)]{FOT11}. By the Markov property, $\|\alpha R_\alpha f\|_\infty\le\|f\|_\infty$ for $f\in L^2(\ambient,\meas)\cap L^\infty(\ambient,\meas)$ \cite[Theorem~1.4.1]{FOT11}.
 
 A central role is played by the energy measures associated with a MMD space, which are defined as follows. Note that $fg\in\domain$ for any $f,g\in\domain\cap L^{\infty}(\ambient,\meas)$ by \cite[Theorem 1.4.2-(ii)]{FOT11} and that $\{(-n)\vee(f\wedge n)\}_{n=1}^{\infty}\subset\domain$ and $\lim_{n\to\infty}(-n)\vee(f\wedge n)=f$ in norm in $(\domain,\form_{1})$ by \cite[Theorem 1.4.2-(iii)]{FOT11}.
 
\begin{definition}[{\cite[(3.2.13), (3.2.14) and (3.2.15)]{FOT11}}]\label{d:EnergyMeas}
 Let $(\ambient,\metric,\meas,\form,\domain)$ be a MMD space. The \emph{energy measure} $\Gamma\langle f,f\rangle$ of $f\in\domain$ associated with $(\ambient,\metric,\meas,\form,\domain)$ is defined, first for $f\in\domain\cap L^{\infty}(\ambient,\meas)$ as the unique ($[0,\infty]$-valued) Borel measure on $\ambient$ such that 
 \begin{equation}\label{e:EnergyMeas}
 		\int_{\ambient} g \dif\Gamma\langle f,f\rangle = \form(f,fg)-\frac{1}{2}\form(f^{2},g) \qquad \textrm{for all $g \in \domain\cap\mathcal{C}_{\mathrm{c}}(\ambient)$,}
 \end{equation}
 and then by $\Gamma\langle f,f\rangle(A):=\lim_{n\to\infty}\Gamma\bigl((-n)\vee(f\wedge n),(-n)\vee(f\wedge n)\bigr)(A)$ for each Borel subset $A$ of $\ambient$ for general $f\in\domain$. We also define the \emph{mutual energy measure} $\Gamma\langle f,g\rangle$ of $f,g\in\domain$ as the Borel signed measure on $\ambient$ given by $\Gamma\langle f,g\rangle:=\frac{1}{2}\bigl(\Gamma\langle f+g,f+g\rangle-\Gamma\langle f,f\rangle-\Gamma\langle g,g\rangle\bigr)$, so that $\Gamma\langle\mspace{-0.8mu}\cdot\mspace{-0.8mu},\mspace{-0.8mu}\cdot\mspace{-0.8mu}\rangle$ is bilinear and symmetric and satisfies the Cauchy--Schwarz inequality:
 	\begin{align}\label{e:bilinear-energy}
 		\Gamma\langle&af+bg,af+bg\rangle=a^{2}\Gamma\langle f,f\rangle+2ab\Gamma\langle f,g\rangle+b^{2}\Gamma\langle g,g\rangle, \qquad a,b \in \mathbb{R},\\
 		|\Gamma&\langle f,g\rangle(B)|^{2} \leq \Gamma\langle f,f\rangle(B)\Gamma\langle g,g\rangle(B) \qquad \textrm{for all Borel subsets $B$ of $\ambient$.}
 		\label{e:Cauchy-Schwarz-energy}
 	\end{align}
 	Note that by \cite[Lemma 3.2.3]{FOT11} and the strong locality of $(\form,\domain)$,
 	\begin{equation}\label{e:EnergyMeasTotalMass}
 		\Gamma\langle f,g\rangle(\ambient)=\form(f,g) \quad \textrm{for all $f,g \in \domain$.}
 	\end{equation}
 \end{definition}

 \begin{definition} \label{d:domain-local}
 Assume that $(\ambient,\meas,\form,\domain)$ is strongly local, and let $D$ be an open subset of $\ambient$. We define the \emph{space $\domain_{\loc}(D)$ of functions on $D$ locally in $\domain$} as
 	\begin{equation}\label{e:Floc}
 		\domain_{\loc}(D) := \Biggl\{ f \Biggm|
 		\begin{minipage}{285pt}
 			$f$ is an $\meas$-equivalence class of $\mathbb{R}$-valued Borel measurable functions on $D$ such that $f = f^{\#}$ $\meas$-a.e.\ on $V$ for some $f^{\#}\in\domain$ for each relatively compact open subset $V$ of $D$
 		\end{minipage}
 		\Biggr\},
 	\end{equation}
 	and define the \emph{$\form$-energy measure} $\Gamma_{D}(f,f)$ of $f \in \domain_{\loc}(D)$ as the unique Radon measure on $D$ such that $\Gamma_{D}\langle f,f\rangle (A)=\Gamma\langle f^{\#},f^{\#}\rangle (A)$ for any relatively compact Borel subset $A$ of $D$ and any $V,f^{\#}$ as in \eqref{e:Floc} with $A\subset V$; note that $\Gamma(f^{\#},f^{\#})(A)$ is independent of a particular choice of such $V,f^{\#}$ by \cite[Corollary 3.2.1]{FOT11}. We abbreviate $\Gamma\langle f,f\rangle$ as $\Gamma\langle f\rangle$.
 \end{definition}

 We recall the notions of capacity and quasi-continuity associated with a regular Dirichlet form; see \cite[Section~2.1]{FOT11}. For an open set $U\subset \ambient$, define
 \[\Capacity_1(U):=\inf\Sett{\form_1(u,u)}{u\in\domain,\ u\ge1\quad m\text{-a.e.\ on }U}.\]
 For an arbitrary set $A\subset \ambient$, let
 \[\Capacity_1(A):=\inf\Sett{\Capacity_1(U)}{A\subset U,\ U\text{ open}}.\]
 A property holds \emph{$\form$-quasi-everywhere} ($\form$-q.e.) if it holds outside a set of zero $\Capacity_1$-capacity.
 
 A function $u$ is called \emph{$\form$-quasi-continuous} if, for every $\varepsilon>0$, there exists an open set $G\subset \ambient$ with $\Capacity_1(G)<\varepsilon$ such that $u|_{\ambient\setminus G}$ is finite and continuous. Every $u\in\domain$ admits a $\form$-quasi-continuous representative $\widetilde u$, unique up to a set of zero capacity; see \cite[Theorem~2.1.3 and Lemma~2.1.4]{FOT11}.
 
 Strong locality implies the Leibniz and chain rules for energy measures; see \cite[Lemma~3.2.5 and Theorem~3.2.2, and the discussion following Theorem~3.2.2]{FOT11}. We recall these identities below. Let $D\subset \ambient$ be open, and let $f,g,h\in\domain_{\loc}(D)\cap L^\infty_{\loc}(D,\meas)$. Then $fg\in\domain_{\loc}(D)$, and
 \begin{equation}\label{e.energy-Leibniz}
 	\dif\Gamma_D\langle fg,h\rangle=\widetilde f\dif \Gamma_D\langle g,h\rangle+\widetilde g\dif \Gamma_D \langle f,h\rangle,
 \end{equation}
 where $\widetilde f$ and $\widetilde g$ denote $\form$-quasi-continuous representatives.
 
 More generally, let $u_1,\ldots,u_n\in\domain_{\loc}(D)\cap L^\infty_{\loc}(D,m)$, and let $\Phi,\Theta\in C^1(\mathbb R^n)$. Then $\Phi(\mathbf u),\Theta(\mathbf u)\in\domain_{\loc}(D)$, where $\mathbf u=(u_1,\ldots,u_n)$, and
 \begin{equation}\label{e.energy-chain}
 	\dif\Gamma_D\bigl\langle\Phi(\mathbf u),\Theta(\mathbf u)\bigr\rangle
 	=
 	\sum_{i,k=1}^n(\partial_i\Phi)(\widetilde{\mathbf u})(\partial_k\Theta)(\widetilde{\mathbf u})\dif \Gamma_D\langle u_i,u_k\rangle.
 \end{equation}
 Here $\widetilde{\mathbf u}$ denotes the vector of $\form$-quasi-continuous representatives. The use of these representatives is justified  by the fact that energy measures do not charge sets of zero capacity \cite[Lemma~3.2.4]{FOT11}. In particular, when the functions $u_i$ are continuous on $D$, the $\form$-quasi-continuous representatives may be replaced by the functions themselves in \eqref{e.energy-Leibniz} and \eqref{e.energy-chain}.

\subsection{Martingale dimension}
We recall the definition of martingale dimension through the index of the Dirichlet form, following  \cite[Section~2]{Hin10}.

A $\sigma$-finite Borel measure $\nu$ on $\ambient$ is called an \emph{energy-dominant measure} if $\Gamma\langle u\rangle\ll\nu$ for all $u\in\domain$. It is called a \emph{minimal energy-dominant measure} if, in addition, $\nu\ll\mu$ for every other energy-dominant $\sigma$-finite Borel measure $\mu$. Such a measure always exists and is unique up to mutual absolute continuity \cite[Lemmas~2.2--2.4]{Hin10}.

Fix a minimal energy-dominant measure $\nu$. For $n\in\mathbb N$ and $u_1,\ldots,u_n\in\domain$, consider the symmetric matrix
\[A_{\mathbf u}(x)=\left(\frac{\dif\Gamma\langle u_i,u_k\rangle}{\dif\nu}(x)\right)_{1\le i,k\le n}.\]
This matrix is positive semidefinite for $\nu$-a.e. $x\in \ambient$.

The \emph{pointwise index} of the Dirichlet form is the smallest, in the $\nu$-almost everywhere sense, measurable function \[p:\ambient\to\mathbb N\cup\{0,\infty\}\] such that, for every $n\in\mathbb N$ and every $u_1,\ldots,u_n\in\domain$, \[\operatorname{rank}A_{\mathbf u}(x)\le p(x)\quad\text{for $\nu$-a.e. }x\in \ambient.\]
By \cite[Proposition~2.10]{Hin10}, if $\{u_k\}_{k\in\mathbb N}$ is a countable family whose linear span is dense in $(\domain,\form_1^{1/2})$, then
\[p(x)=\sup_{n\in\mathbb N}\operatorname{rank}\left(\frac{\dif\Gamma\langle u_i,u_k\rangle}{\dif\nu}(x)\right)_{1\le i,k\le n}\quad\text{for $\nu$-a.e. }x\in \ambient.\]

The \emph{index} of the Dirichlet form is defined by \[\operatorname{index}(\form,\domain)=\operatorname*{\nu-ess\,sup}_{x\in \ambient} p(x).\] This definition is independent of the choice of minimal energy-dominant measure. By \cite[Theorem~3.4]{Hin10}, the index coincides with the probabilistic martingale dimension of the associated symmetric diffusion. We therefore define \emph{martingale dimension} 
\begin{equation}\label{e.martingale-dimension}
	d_{\mathrm m}:=\operatorname*{\nu-ess\,sup}_{x\in \ambient}p(x).
\end{equation}
In particular, for every integer $1\le n\le d_{\mathrm m}$, there exist $u_1,\ldots,u_n\in\domain$ such that
\[\nu\left(\left\{x\in \ambient:\operatorname{rank}\left(\frac{\dif\Gamma\langle u_i,u_k\rangle}{\dif\nu}(x)\right)_{1\le i,k\le n}=n\right\}\right)>0.\]
Moreover, these functions may be selected from any linear subspace of $\domain$ that is dense with respect to $\form_1^{1/2}$; see \cite[Proposition~2.10]{Hin10}. This characterization will be used in Lemma~\ref{l.core} to select functions whose energy-measure matrix has full rank on a set of positive measure.

\subsection{Volume doubling, scale functions, and analytic assumptions}

We recall the geometric and analytic conditions used throughout this paper.  
\begin{definition}\label{d.VD}
Let $(\ambient,\metric,\meas)$ be a metric measure space. For $x\in \ambient$ and $r\in(0,\infty)$, write
	\[B(x,r)=\Sett{y\in \ambient}{d(x,y)<r},\qquad V(x,r)=\meas(B(x,r)).\]
We say that $(\ambient,\metric,\meas)$ satisfies the \hypertarget{VD}{\emph{volume doubling property}}, denoted by $\VD$, if there exists $C_{\mathrm{VD}}\ge1$ such that
\begin{equation}\label{e.VD}
	0<V(x,2r)\le C_{\mathrm{VD}}V(x,r)<\infty,\ \text{ for every $x\in \ambient$ and  $r>0$}
\end{equation}
\end{definition}
By iterating \eqref{e.VD}, there exist constants $C\ge1$ and $\alpha>0$, depending only on $C_{\mathrm{VD}}$, such that
	\begin{equation}\label{e.VD-polynomial}
		\frac{V(x,R)}{V(x,r)}\le C\left(\frac Rr\right)^\alpha,\quad\text{ for every $x\in \ambient$ and } 0<r\le R<\infty.
	\end{equation}
\begin{definition}[Scale functions]\label{d.scale}
	A \emph{scale function} is a continuous increasing bijection $\Psi:[0,\infty)\to[0,\infty)$ such that there exist constants $C_\Psi\ge1$ and $1<\beta\le\beta'<\infty$ satisfying
	\begin{equation}\label{e.psibeta}
		C_\Psi^{-1}\left(\frac Rr\right)^\beta \le \frac{\Psi(R)}{\Psi(r)} \le C_\Psi\left(\frac Rr\right)^{\beta'},\qquad 0<r\le R<\infty.
	\end{equation}
	We refer to \eqref{e.psibeta} as the \emph{weak scaling condition} for $\Psi$.
\end{definition}
	
In particular, if $\Psi$ is a scale function, then $\Psi(0)=0$ and $\Psi(r)>0$ for every $r>0$. The scale function describes the characteristic time scale associated with a spatial distance $r$. The standard example is $\Psi(r)=r^\beta$, in which case \eqref{e.psibeta} holds with $C_\Psi=1$ and $\beta'=\beta$. The case $\Psi(r)=r^2$ corresponds to the usual diffusive scaling, whereas larger exponents arise naturally in sub-Gaussian settings.
\begin{definition}[Poincar\'e inequality]\label{d.PI}
	Let $(\ambient,\metric,\meas,\form,\domain)$ be a MMD space, and let $\Psi$ be a scale function. We say that the space satisfies the \hypertarget{PI}{\emph{Poincar\'e inequality at scale $\Psi$}}, denoted by $\PI{\Psi}$, if there exist constants $C_{\mathrm{PI}}>0$ and $\Lambda_{\mathrm{PI}}\ge1$ such that
	\begin{equation}\label{e.PI}
		\int_{B(x,r)}|u-u_{B(x,r)}|^2\dif \meas\le C_{\mathrm{PI}}\Psi(r)\Gamma\langle u\rangle(B(x,\Lambda_{\mathrm{PI}}r)).
	\end{equation}
	for every ball $B(x,r)$ and every $u\in\domain_{\loc}(B(x,\Lambda_{\mathrm{PI}}r))$, where \[u_{B(x,r)}:=\frac{1}{V(x,r)}\int_{B(x,r)}u\dif \meas.\]
	The constant $\Lambda_{\mathrm{PI}}$ is called the \emph{dilation constant} of the Poincar\'e inequality.
\end{definition}

\begin{definition}[Capacity upper bound]
	\label{d.Cap-upper}
	Let $(\ambient,\metric,\meas,\form,\domain)$ be a MMD space, and let $\Psi$ be a scale function. We say that the space satisfies the \hypertarget{Cap}{\emph{capacity upper bound at scale $\Psi$}}, denoted by $\Capa{\Psi}$, if there exists a constant $C_{\mathrm{Cap}}>0$ such that, for every ball $B(x,r)$, there exists $\varphi\in\domain\cap C_c(\ambient)$ satisfying
	\begin{equation}\label{e.Cap-upper}
	0\le\varphi\le1,\quad \varphi=1\ \text{on }B(x,r),\quad \supp\varphi \subset B(x,2r), \ \text{ and }\	\form(\varphi)\le C_{\mathrm{Cap}}\frac{V(x,r)}{\Psi(r)}.
	\end{equation}
\end{definition} 

Let $(\ambient,\metric,\meas)$ be a metric measure space, and let $\Psi:[0,\infty)\to[0,\infty)$ be a scale function. For $x\in\ambient$, $r\in(0,\infty)$, $\lambda\in(0,1)$, set 
\[D(x,r,\lambda):=\frac{2\log(V(x,r)/V(x,\lambda r))}{\log(\Psi(r)/\Psi(\lambda r))}.\]
Define
\begin{equation}\label{e.ds-unif}
d_{\mathrm s}:=\liminf_{\lambda\downarrow0}\limsup_{r\downarrow0}\sup_{x\in \ambient}D(x,r,\lambda).
\end{equation}

The following is the main result of this paper.
\begin{theorem}\label{t.main}
Let $(\ambient,\metric,\meas,\form,\domain)$ be a metric measure Dirichlet space satisfying $\VD$, $\PI{\Psi}$, and $\Capa{\Psi}$ for a scale function satisfying \eqref{e.psibeta}. Then its martingale dimension satisfies
\begin{equation}\label{e.main-statement}
1\le d_{\mathrm m}\le d_{\mathrm s}<\infty.
\end{equation}
\end{theorem}

\subsection{Related work and historical background}

The study of martingale dimension through energy measures was pioneered by Kusuoka \cite{Kus89,Kus93}, who studied diffusion processes on post-critically finite self-similar fractals. In particular, he proved that the martingale dimension of Brownian motion on the Sierpi\'nski gasket is one, regardless of the dimension of the ambient Euclidean space. His approach used energy measures and products of random matrices to analyze the martingale dimension. Building on Kusuoka's work, Hino developed a general analytic characterization of martingale dimension in terms of the pointwise rank of energy-measure matrices. More precisely, \cite[Theorem~3.4]{Hin10} establishes that the index of a strongly local regular Dirichlet form, defined using a minimal energy-dominant measure, coincides with the martingale dimension of the associated diffusion defined in terms of stochastic integral representations.

Hino subsequently developed a method for obtaining upper bounds on martingale dimension using harmonic functions and Euclidean image measures \cite{Hin13}. In particular, he proved that the martingale dimension is bounded above by the spectral dimension for natural diffusions on Sierpi\'nski carpets. 
In the introduction to that paper, Hino wrote
\cite[p.~741]{Hin13}:
\begin{quote}
	``Although we presently require the self-similar
	structure of the state space for this argument,
	we expect the relation
	$1\le d_{\mathrm m}\le d_{\mathrm s}$
	to be true for more general metric measure
	spaces as well.''
\end{quote}
 The present work resolves this conjecture for strongly local regular Dirichlet forms under the volume doubling, Poincar\'e inequality, capacity upper bound, and heat-kernel assumptions considered here, \emph{without requiring any self-similar structure}.

More recently, Hino \cite{Hin25} developed a different extension of his method beyond self-similar spaces. Under a localized energy--capacity balance condition formulated using refining partitions, he proved that the martingale dimension is one \cite[Theorem~3.3]{Hin25}, without assuming heat kernel estimates or self-similarity. His proof constructs locally harmonic functions whose energy-measure matrices become asymptotically scalar under blow-up, and studies the weak limits of their Euclidean image measures \cite[Proposition~4.1]{Hin25}. The resulting absolute continuity of the image energy measures is used, together with logarithmic cutoff functions, to contradict the assumed lower bounds on relative capacities of points. The method applies, in particular, to inhomogeneous Sierpi\'nski gaskets.

Both approaches use blow-ups and weak limits of Euclidean image measures, thereby avoiding compactness of Dirichlet forms on varying spaces. The main differences lie in the construction of the coordinates and the final contradiction. Hino \cite[Proposition~4.1]{Hin25} works with locally harmonic functions on a refining sequence of partitions, whereas we start with functions that need not be harmonic and use Hyt\"onen's differentiation theory to select blow-up scales at which they become asymptotically harmonic and their energy-measure matrices become asymptotically scalar. Moreover, Hino uses a localized energy--capacity balance condition to show that the martingale dimension is one, while our argument derives uniform Sobolev inequalities from heat-kernel estimates to obtain a quantitative upper bound in terms of the spectral dimension.

A connection between martingale dimension and another notion of dimension was established in \cite[Theorem~3.2]{MurGauss25}. Recall that a metric measure space $(X,d,m)$ is called a \emph{Lipschitz differentiability space}
if it admits a countable family of measurable charts $(U_\alpha,\varphi_\alpha)$, where $U_\alpha\subset X$ are Borel sets covering
$X$ up to an $m$-null set and $\varphi_\alpha:X\to\mathbb R^{n_\alpha}$
are Lipschitz maps, such that every Lipschitz function $f:X\to\mathbb R$ is differentiable almost everywhere with respect to these charts.
More precisely, for $m$-almost every $x\in U_\alpha$, there exists a unique vector
$D_\alpha f(x)\in\mathbb R^{n_\alpha}$ satisfying
\[
\lim_{y\to x} \frac{\left| 	f(y)-f(x) 	-D_\alpha f(x)\cdot (\varphi_\alpha(y)-\varphi_\alpha(x)) \right|}{d(x,y)}=0.
\]
The essential supremum of the chart dimensions $n_\alpha$ is called the \emph{analytic dimension}. This notion generalizes Rademacher's
differentiability theorem to metric measure spaces; see \cite[Section~4]{Che99} and \cite[Definition~3.1]{Bat15}.
Cheeger \cite[Theorem~4.38]{Che99} proved that doubling metric measure spaces supporting a suitable Poincar\'e inequality admit such a differentiable structure.

Under two-sided Gaussian heat kernel bounds, the underlying metric measure space is a Lipschitz differentiability space. Moreover, the pointwise index of Hino coincides almost everywhere with the pointwise dimension of Cheeger's measurable differentiable structure \cite[Theorem~3.2]{MurGauss25}.
Consequently, the martingale dimension coincides with the analytic dimension.
 This gives the upper bound \[d_{\mathrm m}\le \dim_{\mathrm H}(\ambient,\metric)\]
under Gaussian heat kernel bounds \cite[Corollary~3.3]{MurGauss25}, where $\dim_{\mathrm H}(\ambient,\metric)$ denotes the Hausdorff dimension.

However, this approach does not extend to the strictly sub-Gaussian setting. Recent work of Eriksson-Bique and the second-named author \cite[Theorem~4.1]{EBM26} shows that if a strongly local regular Dirichlet form satisfies two-sided sub-Gaussian heat kernel estimates with walk dimension $\beta>2$, then the underlying metric measure space is not a Lipschitz differentiability space. In particular, Cheeger's analytic dimension, defined through a measurable differentiable structure, is not available in this setting. This highlights the need for a different approach to bounding martingale dimension under sub-Gaussian heat kernel estimates, as developed in the present paper.

This connection in the Gaussian case also reveals a close analogy between Hino's dimension bound and a conjecture of Cheeger. Cheeger \cite[Conjecture~4.63]{Che99} conjectured that the image of an $n$-dimensional differentiability chart has positive $n$-dimensional Hausdorff measure. In particular, this would imply that the analytic dimension is bounded above by the Hausdorff dimension of the underlying space. Under Gaussian heat kernel bounds, the equality between martingale and analytic dimensions identifies this consequence of Cheeger's conjecture with the corresponding upper bound for martingale dimension. In the presence of Ahlfors regular volume growth and Gaussian heat kernel estimates, the Hausdorff and spectral dimensions coincide, making the analogy with Hino's spectral-dimension bound particularly direct.

Cheeger's conjecture was resolved by De Philippis, Marchese and Rindler \cite[Theorem~4.1.1]{DPMR17}, who proved the stronger statement that the pushforward of the reference measure under a differentiability chart is absolutely continuous with respect to Lebesgue measure. Their proof uses independent Alberti representations and a structure theorem for normal currents due to De Philippis and Rindler \cite{DR16}.

The energy image density conjecture of Bouleau and Hirsch \cite{BH91} was recently resolved by Eriksson-Bique and the second-named author, who also extended
the result to strongly local regular Dirichlet forms without assuming absolute continuity of energy measures \cite[Theorems~1.7 and~2.65]{EBM25}. This work further establishes finiteness of martingale dimension under sub-Gaussian heat kernel estimates \cite[Theorem~5.7]{EBM25}, resolving \cite[Conjecture~3.12]{MurGauss25}, and provides a new proof of Cheeger's conjecture \cite[Proposition~5.18]{EBM25}.

The present work pursues a complementary, quantitative question: bounding martingale dimension by the spectral dimension rather than merely establishing its finiteness. Our approach returns to Hino's method of harmonic coordinates and Euclidean image measures, but replaces the self-similar compactness arguments by a blow-up procedure that does not require convergence of the underlying Dirichlet forms.

\subsection{Proof sketch and inspiration}

The main idea of our proof originates in Hino's work on martingale dimension for self-similar fractals \cite{Hin13}. We begin by recalling his argument in an idealized setting, both to explain the mechanism behind the dimension bound and to highlight the difficulties in extending this approach to general metric measure Dirichlet spaces.

Suppose that there exist $n$ harmonic functions $H^1,\ldots,H^n$ whose energy-measure matrix is exactly scalar:
\[\Gamma\langle H^i,H^k\rangle =\delta_{ik}\nu.\]
Writing $\mathbf H=(H^1,\ldots,H^n)$, the chain rule gives
\[\Gamma(u\circ\mathbf H)=|\nabla u|^2(\mathbf H)\,\nu,\ \text{for smooth functions $u$ on $\mathbb R^n$}.\]
Thus the energy of functions of the coordinates takes the form of a weighted Euclidean Dirichlet energy. This viewpoint is closely related to the theory of image Dirichlet structures developed by Bouleau and Hirsch \cite[Chapter~V]{BH91}, in which a measurable map transports a Dirichlet structure to its image, yielding a Dirichlet structure on Euclidean space. In our setting, however, the energy measures need not be absolutely continuous with respect to the reference measure, and we work directly with their Euclidean pushforwards rather than constructing an image Dirichlet form.

Let $\zeta$ be a suitable cutoff function, equal to one on a smaller region $V$, and consider the Euclidean image measures
\[\kappa=\mathbf H_*(\zeta^2\nu),\qquad\sigma=\mathbf H_*(\one_Vm).\]
The measure $\kappa$ governs the energy of functions of the coordinates, while $\sigma$ governs their $L^q$ norms.  Since these measures need not be mutually absolutely continuous, the main challenge is to obtain suitable growth estimates for both measures at a common point. This is achieved using Sobolev regularity, capacity estimates, and the pushforward Poincar\'e inequality.

Following the integration-by-parts argument of Hino \cite[Proposition~3.7]{Hin13}, harmonicity and the scalar energy-measure identity imply that
\[\kappa=\rho\mathcal L^n,\ \text{ and }\ \sqrt\rho\in H^1(\mathbb R^n).\]
The key observation is that integration by parts, combined with a Cauchy--Schwarz inequality for energy measures, gives a distributional identity $\nabla\kappa=b\kappa$ with $b\in L^2(\kappa;\mathbb R^n)$, from which the Sobolev regularity of $\sqrt\rho$ follows. The corresponding argument for approximate coordinates is carried out in Lemma~\ref{l.euclidean-image}-\ref{it.image2}, using the matrix-valued Cauchy--Schwarz inequality of Lemma~\ref{l.matrix-CS}.

A pushforward Poincar\'e inequality provides additional control of $\sigma$. Together with the nondegeneracy of the coordinates, which rules out $\sigma$ being a point mass, and the Sobolev regularity of $\sqrt\rho$, this implies that $\sigma$ does not charge sets of zero $(1,2)$-Sobolev capacity. These conclusions are established for the limiting image measures in Lemma~\ref{l.image-poincare}-\ref{it.image-poincare2}.

Now suppose that the original space satisfies a Sobolev inequality of dimension $s$, where $2<s<n$. Applying this inequality to functions of the form $u\circ\mathbf H$, with suitable cutoffs, yields a weighted Sobolev inequality on $\mathbb R^n$ of the form
\[\left(\int_{\mathbb R^n}|u|^q\dif \sigma\right)^{2/q}\le C\left(\int_{\mathbb R^n}|\nabla u|^2\dif \kappa+\int_{\mathbb R^n}u^2\dif \lambda\right),\]
where $q=2s/(s-2)$ and $\lambda$ is a finite measure arising from the cutoff energy and the $L^2$ term.

The regularity of $\sqrt\rho$, together with the non-concentration of $\sigma$ on polar sets and the Euclidean capacity and density estimates collected in Proposition~\ref{p.CapE}, allows us to find a point $a\in\mathbb R^n$ such that
\[
\kappa(B(a,r))\le Cr^n,\qquad
\sigma(B(a,r))\ge cr^n,\ \text{ and }\
\lambda(B(a,r))=o(r^{n-2}),\ \text{ as }r\downarrow0.
\]
The upper bound for $\kappa$ controls the energy of Euclidean cutoff functions, while the lower bound for $\sigma$ controls their $L^q$ norms from below. The contribution of $\lambda$ is negligible at a relevant scale. Testing the weighted Sobolev inequality with cutoff functions supported on balls of radius $r$ leads to \[r^{2n/q}=r^{n(s-2)/s}\lesssim r^{n-2},\] which is impossible as $r\downarrow0$ since $n>s$. When $s<2$, an analogous contradiction is obtained using an $L^\infty$ Sobolev inequality and logarithmic cutoff functions in $\mathbb R^2$. This is essentially the dimension contradiction underlying Hino's approach; see \cite[Theorem~3.14 and Strategy~3.15]{Hin13}.

The main difficulty in implementing the preceding argument is the construction of harmonic coordinates with an exactly scalar energy-measure matrix. Indeed, the assumption that the martingale dimension is at least $n$ only guarantees the existence of $n$ functions whose energy-measure matrix has rank $n$ on a set of positive measure. These functions need not be harmonic, nor need their energy-measure matrix be scalar, as required for the preceding argument.

In his earlier work on self-similar fractals, Hino \cite{Hin13} exploited two important structural advantages. First, the space of piecewise harmonic functions is dense in the domain of the Dirichlet form; see the discussion preceding \cite[Proposition~4.7]{Hin13}, which refers to \cite[Lemma~3.10]{Hin05}. This allows him to select locally harmonic functions whose energy-measure matrix has full rank on a set of positive measure. Second, self-similarity allows harmonic functions on shrinking cells to be pulled back to a fixed Dirichlet space; see \cite[Lemma~4.3 and Corollary~4.6]{Hin13}. In the case of Sierpi\'nski carpets, compactness of harmonic functions \cite[Proposition~4.16]{Hin13} then leads to limiting harmonic coordinates with an exactly constant energy-measure matrix, as in the proof of \cite[Theorem~4.14]{Hin13}.

More recently, Hino \cite{Hin25} developed a related approach that no longer requires self-similarity. Under assumptions involving refining partitions and a localized energy--capacity balance condition, he constructs locally harmonic coordinates and studies weak limits of their Euclidean image measures; see \cite[Propositions~3.10 and 4.1]{Hin25}. In particular, this approach does not require convergence of Dirichlet forms on varying spaces. Thus an important feature shared by his recent work and ours is that the limiting analysis takes place on a fixed Euclidean space. The assumptions of \cite{Hin25} and those of the present work are different, and neither set of assumptions subsumes the other. Hino's localized energy--capacity balance condition, together with his refining partition hypotheses, leads to the conclusion that the martingale dimension is one \cite[Theorem~3.3]{Hin25}. By contrast, our assumptions allow arbitrarily large martingale dimension, as illustrated by Brownian motion on $\mathbb R^n$. Our goal is therefore to obtain a quantitative upper bound on martingale dimension rather than a sufficient condition to establish that it equals one.

One could also seek limiting harmonic coordinates through a compactness theory for Dirichlet forms on varying metric measure spaces, perhaps using the framework of Mosco convergence developed by Kuwae and Shioya \cite{KS03}. Recent work of the first-named author \cite[Theorem~1.2]{Chen26} establishes Mosco convergence of Dirichlet forms under uniform heat-kernel estimates and additional geometric assumptions, including homogeneous volume growth and the requirement that the underlying spaces be proper length spaces (and hence satisfy a chain condition). Neither homogeneous volume growth nor the chain condition is assumed in our setting. Moreover, retaining sufficient control of the energy-measure matrices of the coordinate functions under such convergence could require further arguments.

A further inspiration comes from Cheeger's work on differentiability of Lipschitz functions on metric measure spaces. Cheeger \cite[Theorem~3.7]{Che99} showed that, under a Vitali covering assumption, Lipschitz functions are asymptotically generalized linear at almost every point. In particular, they are asymptotically $p$-harmonic, in the sense that their normalized $p$-energy becomes asymptotically minimal under perturbations with fixed boundary values. Thus, even functions that are not harmonic may exhibit asymptotic harmonicity under blow-up, that is, when examined at increasingly small spatial scales. Although Cheeger's result concerns a different notion of energy and cannot be applied directly to our Dirichlet forms, it provides a natural motivation for seeking harmonicity only after blow-up.

Our approach is to construct approximate harmonic coordinates without first selecting locally harmonic functions. Starting from a dense resolvent core of functions with bounded generators, we use Hyt\"onen's differentiation theory \cite{Hyt10} to select blow-up scales at which their energy-measure matrix becomes approximately constant. After a suitable linear transformation and normalization, this matrix becomes asymptotically scalar, while the generators of the rescaled functions tend to zero. These functions, which we call \emph{approximate harmonic coordinates}, are constructed in Proposition~\ref{p.finite-normalization}.

As in Hino's recent work, we then pass to weak limits of Euclidean image measures rather than seeking convergence of the coordinate functions or the underlying Dirichlet forms. The main additional difficulty is that our coordinates are only approximately harmonic and their energy-measure matrices are only approximately scalar. We show that the corresponding errors vanish in the limit, recovering the integration-by-parts identity and Sobolev regularity of the limiting image energy measure.

Finally, the contradiction argument reflects the different objectives of the two approaches. Hino's recent work uses a localized energy--capacity balance condition to rule out martingale dimension greater than one. In contrast, we derive uniform Sobolev inequalities from heat-kernel estimates and combine them with the Euclidean image-measure analysis to obtain the quantitative bound $d_{\mathrm m}\le d_{\mathrm s}$.

\subsection{Organization of the paper}

In Section~\ref{s.scale}, we make the construction of approximate harmonic coordinates rigorous. Lemma~\ref{l.core}-\ref{it.core3}, based on Hino's characterization of martingale dimension, provides functions whose energy-measure matrix has rank at least $n$ on a set of positive measure. Hyt\"onen's differentiation theory, in the form of Lemma~\ref{l.diffe}, allows us to select a point and a sequence of good blow-up scales at which the normalized energy-measure matrix is approximately constant, while the reference measure is suitably controlled relative to the energy measure. Proposition~\ref{p.finite-normalization} then constructs approximate harmonic coordinates whose energy-measure matrices are asymptotically scalar, with errors tending to zero in total variation. These functions are uniformly bounded and H\"older continuous on appropriate rescaled balls, and their generators tend to zero uniformly.

Section~\ref{s.image} studies the Euclidean image measures associated with these coordinates. Lemma~\ref{l.uniform-cutoff} provides cutoff functions with uniformly bounded energies. Following Hino's integration-by-parts method, Lemma~\ref{l.euclidean-image}-\ref{it.image2} establishes that the weak limit of the image energy measures has the form $\kappa=\rho\mathcal L^n$, with $\sqrt\rho\in H^1(\mathbb R^n)$. The main difficulty is that the coordinates are only approximately harmonic and their energy-measure matrices are only  approximately scalar. The matrix-valued Cauchy--Schwarz inequality of Lemma~\ref{l.matrix-CS} allows us to control the resulting errors and obtain the required integration-by-parts identity in the limit. Lemma~\ref{l.image-poincare}  studies the weak limit $\sigma$ of the pushforwards of the reference measures. A pushforward Poincar\'e inequality, together with a Caccioppoli argument excluding point masses, shows that $\sigma$ does not charge sets of zero $(1,2)$-Sobolev capacity. These results, together with the Euclidean capacity estimates of Proposition~\ref{p.CapE}, provide the measure-theoretic ingredients needed for the contradiction described above.

Finally, Section~\ref{s.proof} establishes the required uniform Sobolev inequalities and completes the proof. Lemma~\ref{l.uniform-profile} extracts uniform volume growth estimates from the definition of the uniform small-scale spectral dimension. Combined with heat-kernel upper bounds, these give uniform ultracontractivity estimates for semigroups killed on fixed rescaled balls (Lemma~\ref{l.local-killed}). The results of Varopoulos and Coulhon then yield the uniform Sobolev inequalities of Lemma~\ref{l.truncated-sob}. Applying these inequalities to functions of the approximate coordinates and passing to the limit gives weighted Sobolev inequalities for the Euclidean image measures. When $d_{\mathrm s}<2$, we use an $L^\infty$-Sobolev inequality and logarithmic cutoff functions to rule out martingale dimension at least two. When $d_{\mathrm s}\ge2$, we use an $L^q$-Sobolev inequality and rescaled Euclidean cutoff functions to obtain a contradiction whenever $d_{\mathrm m}>d_{\mathrm s}$. This completes the proof of Theorem~\ref{t.main}.
\begin{notation}
	Throughout this paper, we use the following notation and conventions.
\begin{enumerate}[label=\textup{({\roman*})}, align=right, leftmargin=*, topsep=5pt, parsep=0pt, itemsep=2pt]
\item $\bN:=\Sett{n\in\bZ}{n\geq1}$.
\item The symbols $\subset$ and $\supset$ allow the case of the equality of two sets.
\item Let $\{(X_{j},d_{j})\}_{j}$ be metric spaces. For each $j$ and $(x,r)\in X_{j}\times (0,\infty)$, we write the open ball in $X_{j}$ centered at $x$ with radii $r$ \[B_{j}(x,r):=\{z\in X_{j}:d_{j}(x,z)<r\}.\]
\item Let $X$ be a non-empty set. We define $\one_{A}=\one_{A}^{X}\in\mathbb{R}^{X}$ for $A\subset X$ by
	 \[\one_{A}(x):=\one_{A}^{X}(x):= \begin{cases}
	 	1 & \mbox{if $x \in A$,}\\
	 	0 & \mbox{if $x \notin A$.}
	 \end{cases} \]
\item Let $X$ be a topological space. We set $C(X):=\{f\mid\textrm{$f:X\to\mathbb{R}$, $f$ is continuous}\}$, $C_{b}(X):=\{f\in C(X)\mid\textrm{$f$ is bounded}\}$, and \[C_c(X):=\{f\in C(X)\mid\textrm{$X\setminus f^{-1}(0)$ has compact closure in $X$}\}.\]
\item We use the letters $C$, $c$ etc. to denote positive constant whose value is inessential and may change at each occurrence. 
\item For two extended real numbers $A,B\in\mathbb{R}\cup\{-\infty,\infty\}$, let $A\wedge B:=\min(A,B)$ and $A\vee B:=\max(A,B)$.
\end{enumerate}
\end{notation}
\section{Selection of good blow-up scales}\label{s.scale}
The goal of this section is to construct a sequence of continuous maps $\mathbf H_j:\ambient\to\mathbb R^n$, whose components belong to $\domain_{\loc}(\ambient)$, such that their energy-measure matrices become asymptotically scalar and the functions become asymptotically harmonic under blow-up. We call these functions \emph{approximate harmonic coordinates}.

\subsection{Differentiation of measures}

\begin{definition}
	We say a metric space $(\ambient,\metric)$ is \hypertarget{MD}{\emph{metric doubling}} $\MD$, if there exists $N\in\bN$ such that, for all $(x,r)\in\ambient\times(0,\infty)$, there exists $\{x_{j}\}_{j=1}^{N}\subset\ambient$ such that $B(x,r)\subset\bigcup_{j=1}^{N}B(x_{j},r/2)$.
\end{definition}
By \eqref{e.VD-polynomial}, for a metric space $(\ambient,\metric)$, if there is a measure such that $(\ambient,\metric,\meas)$ is $\VD$, then $(\ambient,\metric)$ is $\MD$.

We collect some consequences of the differentiation theory for Radon measures on metric doubling spaces, developed by Hyt\"onen \cite[Section~3]{Hyt10}. A fundamental observation of Hyt\"onen is that, even when a Radon measure is not doubling, the metric doubling property guarantees the existence of arbitrarily small scales at almost every point along which the measure satisfies a uniform doubling estimate. Moreover, Hyt\"onen established a Lebesgue differentiation theorem along these good scales. His results play a crucial role in our argument, allowing us to work with energy measures that need not be doubling.

In particular, Hyt\"onen's differentiation theory provides a substitute for the martingale differentiation argument used by Hino on nested self-similar cells. It enables us to select blow-up scales at which the energy-measure matrix is approximately constant, while simultaneously controlling the reference measure relative to the energy measure. The following lemma records the precise consequences needed for our purposes, including the selection of good radii from a prescribed sequence and comparison with a second Radon measure. Its proof combines Hyt\"onen's covering and differentiation arguments with a weak-type maximal estimate obtained using the $5B$-covering lemma.

\begin{lemma}\label{l.diffe}
Let $(\ambient,\metric)$ be a metric space satisfying $\MD$ with constant $N\in\mathbb N$, and let $\nu$ be a positive Radon measure on $\ambient$ such that $0<\nu(\ambient)<\infty$. Fix $K\in[1,\infty)$ and set $D=D_K:=2N^{\lceil \log_{2}(10K)\rceil}$. Define
\begin{equation}
R_\nu(x):=\Sett{r\in(0,\infty)}{\nu(B(x,5Kr))\leq D\nu(B(x,r))},\ \text{ for }x\in\supp[\nu].
\end{equation}
Let $\{\rho_{j}\}_{j\in\bN}$ be a sequence such that $0<\rho_{j+1}\leq\rho_{j}$ for all $j\in\bN$ and $\lim_{j\to\infty}\rho_{j}=0$. Then the following statements hold.
\begin{enumerate}[label=\textup{(\arabic*)},align=right,leftmargin=*,
topsep=5pt,parsep=0pt,itemsep=5pt]
\item\label{it.diffe1} For $\nu$-a.e. $x\in \ambient$, there exists $k_{j}(x)\in\bN$, $j\in\bN$, such that $k_{j}(x)<k_{j+1}(x)$ and $r_{j}:=\rho_{k_{j}(x)}\in R_{\nu}(x)$ for all $j\in\bN$. In particular, $\{r_j\}_{j\in\mathbb N}\subset R_\nu(x)$ and $r_j\downarrow0$.
\item\label{it.diffe2}
For every $n\in\mathbb N$ and every $A=(A_{ij})_{1\le i,j\le n}\in L^1(\ambient,\nu;\mathbb R^{n\times n})$,
\begin{equation}\label{e.diffe1}
	\lim_{\substack{r\downarrow0\\r\in R_\nu(x)}}\frac{\int_{B(x,Kr)}\|A(x)-A(y)\|\dif\nu(y)}{\nu(B(x,r))}=0\quad\text{for }\nu\text{-a.e. }x\in \ambient,
\end{equation}
where, for $B=(B_{ij})\in\mathbb R^{n\times n}$, we set $\|B\|=\sum_{i,j=1}^n|B_{ij}|$.

\item\label{it.diffe3} For every locally finite positive Radon measure $\mu$ on $\ambient$, for $\nu$-a.e. $x\in \ambient$ there exists $r_x\in(0,\infty)$ such that
\begin{equation}\label{e.diffe2}
\sup_{{r\in R_\nu(x)\cap (0,r_{x})}}\frac{\mu(B(x,Kr))}{\nu(B(x,r))}<\infty.
\end{equation}
If, in addition, $\mu(\ambient)<\infty$, then
\begin{equation}\label{e.diffe-gb}
\sup_{r\in R_\nu(x)}\frac{\mu(B(x,Kr))}{\nu(B(x,r))}<\infty\quad\text{for }\nu\text{-a.e. }x\in \ambient.
\end{equation}
\end{enumerate}
\end{lemma}

\begin{proof}
\begin{enumerate}[label=\textup{({\arabic*})}, align=right, leftmargin=*, topsep=5pt, parsep=0pt, itemsep=2pt]
\item[\ref{it.diffe1}] 
We adapt the argument of \cite[Lemma 3.3]{Hyt10}. Define
\begin{equation}
	F_{j}:=\Sett{x\in \supp[\nu]}{\nu(B(x,5K\rho_{j}))>D\nu(B(x,\rho_{j}))},\ j\in\bN,\ \text{ and }N_{0}:=\bigcup_{m\in\bN}\bigcap_{j=m}^{\infty}F_{j}.
\end{equation}
Then $F_{j}$ and $N_{0}$ are Borel sets, since $x\mapsto \nu(B(x,s))$ is lower semicontinuous for each $s\in(0,\infty)$. 

Suppose $\nu(N_{0})>0$. Then by the inner regularity of $\nu$, there exist $m\in\bN$ and a compact $F\subset\ambient$ such that $F\subset\bigcap_{j=m}^{\infty}F_{j}$ and $\nu(F)\in(0,\infty)$. Then 
\begin{equation}\label{e.subse1}
	\nu(B(x,\rho_{j}))\leq D^{-1}\nu(B(x,5K\rho_{j})),\ \text{ for every $j\geq m$ and every $x\in F$}.
\end{equation} 
For any $j\in\bN$, let $Y_{j}\subset F$ be a finite maximal $\rho_{j}$-separated set, that is, $F\subset\bigcup_{y\in Y_{j}}B(y,\rho_{j})$, and for any $y,y'\in Y_{j}$ with $y\neq y'$, we have $\metric(y,y')\geq\rho_{j}$. By $\MD$, we know that 
\begin{equation}\label{e.subse2}
	\#(Y_{j}\cap B(z,5K\rho_{j}))\leq N^{\lceil \log_{2}(10K)\rceil},\ \text{ for every }z\in\ambient.
\end{equation}
Consequently, 
\begin{align}
	0<\nu(F)&\leq \sum_{y\in Y_{j}}\nu(B(y,\rho_{j}))\overset{\eqref{e.subse1}}{\leq}D^{-1}\sum_{y\in Y_{j}}\nu(B(y,5K\rho_{j}))\\
	&\overset{\eqref{e.subse2}}{\leq}N^{\lceil \log_{2}(10K)\rceil}D^{-1}\nu(\Sett{z\in\ambient}{\dist(z,F)<5K\rho_{j}}.
\end{align}
Letting $j\to\infty$ and recalling the definition of $D$, we have $0<\nu(F)\leq \frac{1}{2}\nu(F)$, a contradiction. Therefore, we must have $\nu(N_{0})=0$. By definition of $\{F_{j}\}_{j\in\bN}$, for any $x\in\supp[\nu]\setminus N_{0}$, we have $\#\Sett{j\in\bN}{\rho_{j}\in R_{\nu}(x)}=\infty$. Consequently, we may extract a subsequence of $\{\rho_{k_{j}(x)}\}_{j\in\bN}\subset R_{\nu}(x)$. 

\item[\ref{it.diffe2}] 
Since $D>5^{\log_2N}$, by applying \cite[Corollary~3.6]{Hyt10} to the each entry of $A$, we have,
\begin{equation}\label{e.diffesc}
\lim_{\substack{s\downarrow0\\ \nu(B(x,5s))\leq D\nu(B(x,s))}}\fint_{B(x,s)}|A_{ij}(x)-A_{ij}(y)|\dif\nu(y)=0,
\end{equation}
for each $1\le i,j\le n$ and $\nu$-almost every $x\in\ambient$. For every $r\in R_\nu(x)$, we have by the definition of $R_\nu(x)$,
\begin{equation}\label{e.diffodb}
\nu(B(x,5Kr))\leq D\nu(B(x,r))\leq D\nu(B(x,Kr)).
\end{equation}
Therefore, $\Sett{Kr}{r\in R_\nu(x)}\subset \Sett{s\in(0,\infty)}{\nu(B(x,5s))\leq D\nu(B(x,s))}$, and this combines with \eqref{e.diffesc} gives \begin{equation}
	\lim_{\substack{r\downarrow0\\r\in R_\nu(x)}}\fint_{B(x,Kr)}|A_{ij}(x)-A_{ij}(y)|\dif\nu(y)=0,
\end{equation}
for each $1\le i,j\le n$ and $\nu$-a.e. $x\in\ambient$. By the definition of the matrix norm and \eqref{e.diffodb},
\begin{align}
	&\phantom{\ \leq}\lim_{\substack{r\downarrow0\\r\in R_\nu(x)}}
	\frac{\int_{B(x,Kr)}
		\|A(x)-A(y)\|\dif\nu(y)}
	{\nu(B(x,r))}\\
	&\le
	D\sum_{i,j=1}^n
	\lim_{\substack{r\downarrow0\\r\in R_\nu(x)}}
	\fint_{B(x,Kr)}
	|A_{ij}(x)-A_{ij}(y)|\dif\nu(y)\overset{\eqref{e.diffodb}}{=}0.
\end{align}

\item[\ref{it.diffe3}] 
Define
\begin{equation}\label{e.diffem1}
\sM_R[\mu](x):=\sup_{\substack{z\in B(x,s),\ 0<s\leq R\\  \nu(B(z,5s))>0}}\frac{\mu(B(z,s))}{\nu(B(z,5s))},\ \text{ for }(x,R)\in\ambient\times(0,\infty).
\end{equation}
where we adopt the convention $\sup\emptyset:=0$. For every $R,t\in(0,\infty)$, we have
\begin{equation}\label{e.diffem2}
\Sett{x\in \ambient}{\sM_R[\mu](x)>t}
=
\bigcup_{\substack{z\in \ambient,\ 0<s\leq R\\
0<\nu(B(z,5s))<t^{-1}\mu(B(z,s))}}B(z,s).
\end{equation}
Thus $x\mapsto\sM_R[\mu](x)$ is lower semicontinuous, and in particular Borel measurable. By applying the $5B$-covering lemma to the family on the right-hand side of \eqref{e.diffem2}, we obtain an at most countable pairwise disjoint balls $\{B(z_i,s_i)\}_{i\in I}$ such that, $s_{i}\in(0,R]$ for all $i\in I$,
\begin{align}
\Sett{x\in \ambient}{\sM_R[\mu](x)>t}&\subset\bigcup_{i\in I}B(z_i,5s_i),\label{e.diffem31}\\
0<\nu(B(z_i,5s_i))&<t^{-1}\mu(B(z_i,s_i)),\ \text{ for all } i\in I.\label{e.diffem32}
\end{align}

Therefore,
\begin{equation}
\nu(\{\sM_R[\mu]>t\})\overset{\eqref{e.diffem31}}{\leq}\sum_{i\in I}\nu(B(z_i,5s_i))\\ \overset{\eqref{e.diffem32}}{\leq}\frac{1}{t}\sum_{i\in I}\mu(B(z_i,s_i))\leq\frac{\mu(\ambient)}t.
\end{equation}
Set $\sM[\mu]:=\sup_{j\in\mathbb N}\sM_j[\mu]$. Since $\sM_j[\mu]\uparrow\sM[\mu]$,
\begin{equation}\label{e.diffem4}
\nu(\{\sM[\mu]>t\})=\lim_{j\to\infty}\nu(\{\sM_j\mu>t\})\leq\frac{\mu(\ambient)}t,\ \text{ for all } t\in(0,\infty)
\end{equation}
\begin{itemize}
\item If $\mu(\ambient)<\infty$, we have $\nu(\{\sM[\mu]=\infty\})=0$ by \eqref{e.diffem4}. For $x\in\supp[\nu]$ with $\sM[\mu](x)<\infty$ and every $r\in R_\nu(x)$, the ball $B(x,Kr)$ is admissible in \eqref{e.diffem1} whenever $R\geq Kr$, and
	\begin{equation}
		\frac{\mu(B(x,Kr))}{\nu(B(x,r))}
		\leq D\frac{\mu(B(x,Kr))}{\nu(B(x,5Kr))}
		\leq D\sM\mu(x)<\infty.
	\end{equation}
Taking the supremum over $r\in R_\nu(x)$ proves \eqref{e.diffe-gb}.
\item If $\mu$ is locally finite. Since $(\ambient,\metric)$ is second countable (since it is $\MD$, it is separable), there is a countable open cover $\{U_j\}_{j\in\mathbb N}$ such that $\mu(U_j)<\infty$. Let $\mu_j:=\one_{U_j}\cdot {\mu}$ and
\begin{equation}
E:=\bigcup_{j\in\mathbb N}\Sett{x\in \ambient}{\sM[\mu_j](x)=\infty}.
\end{equation}
By \eqref{e.diffem4}, $E$ is Borel measurable and $\nu(E)=0$. Fix $x\in\supp[\nu]\setminus E$, choose $j$ such that $x\in U_j$, and choose $r_x\in(0,\infty)$ such that $B(x,Kr_x)\subset U_j$. Then, for every $r\in R_\nu(x)\cap(0,r_x)$,
\begin{equation}
\frac{\mu(B(x,Kr))}{\nu(B(x,r))}=\frac{\mu_j(B(x,Kr))}{\nu(B(x,r))}\leq D\sM[\mu_j](x)<\infty.
\end{equation}
Taking the supremum over $r\in R_\nu(x)\cap(0,r_x)$ proves \eqref{e.diffe2}.\end{itemize}
\end{enumerate}
\end{proof}

\subsection{Regularity estimates for Poisson equations}
 
 We establish uniform local boundedness and H\"older regularity estimates for solutions of Poisson equations on families of metric measure Dirichlet spaces. The main ingredients are Green operator estimates and local regularity estimates for harmonic functions, which follow from $\VD$, $\PI{\Psi}$ and $\Capa{\Psi}$. The Poisson estimates are then obtained by decomposing a solution into a harmonic function and a Green potential. The uniformity of these estimates is essential for constructing approximate harmonic coordinates in Proposition~\ref{p.finite-normalization}.
  
We first recall the definition of the part Dirichlet form and its Green operator.  For an open set $U\subset X$, define
\[
\domain_0(U):= \left\{
u\in\domain: \widetilde u=0 \quad\text{q.e.\ on }X\setminus U \right\},
\]
where $\widetilde u$ denotes the quasi-continuous representative of $u$.
The restriction of $\form$ to $\domain_0(U)$
defines the part Dirichlet form on $L^2(U,m)$;
see \cite[Section~4.4]{FOT11}.  The restriction of $\form$ to $\domain_0(U)$ defines the part Dirichlet form on $L^2(U,\meas)$, with semigroup $P_t^U$ and nonpositive self-adjoint generator $\gen^U$; see \cite[Section~4.4]{FOT11}.
  
 Define \[\lambda_1(U):=\inf_{f\in\domain_0(U)\setminus\{0\}}\frac{\form(f)}{\|f\|_{L^2(U,\meas)}^2}.\] Whenever $\lambda_1(U)>0$, the \emph{0-order resolvent}, or \emph{Green operator}, is the bounded operator on $L^2(U,m)$ given by \[G_U:=(-\gen^U)^{-1}=\int_0^\infty P_t^U\dif t,\]
 where the integral converges in the strong operator sense. Equivalently, for $f\in L^2(U,\meas)$, $G_{U}f$ is the unique element of $\domain_0(U)$ satisfying
 \[\form(G_Uf,\varphi)=\int_U f\varphi\dif \meas,\ \text{ for all }\ \varphi\in\domain_0(U).\]
 See \cite[Section~4.4]{FOT11} and \cite[(1.22)--(1.25)]{GHL15}.
 
 For the balls considered below, the \emph{Faber--Krahn inequality} available under our assumptions ensures a positive spectral gap for the part form, and hence the existence of its Green operator. The following lemma collects the uniform Green operator and harmonic regularity estimates needed for the subsequent Poisson equation estimates.

Recall that a function $h\in\domain_{\loc}(U)$ is called \emph{$\form$-harmonic} in an open set $U$ if
\[\form(h,\varphi)=0\ \text{ for all }\ \varphi\in\domain\cap C_c(U).\]
It is called \emph{$\form$-subharmonic} in $U$ if
\[\form(h,\varphi)\le0\ \text{ for all nonnegative }\\varphi\in\domain\cap C_c(U).\]

\begin{lemma}\label{l.uniform-analytic}
Let $\{(\ambient_j,\metric_{j},\meas_{j},\form^{(j)},\domain_j,\Psi_j)\}_{j\in\bN}$ be a family of MMD spaces satisfying $\VD$, $\PI{\Psi_j}$ and $\Capa{\Psi_{j}}$, all with the same constants, and suppose all $\Psi_j$ satisfy \eqref{e.psibeta} with the same constants. There are $A_0>1$, $C_0<\infty$ and $\eta\in(0,1]$, independent of $j$, such that on all balls $B_j(z,s)$ with $4A_0s<\diam(\ambient_j,\metric_{j})$:
\begin{enumerate}[label=\textup{({\roman*})}, align=right, leftmargin=*, topsep=5pt, parsep=0pt, itemsep=2pt]
\item\label{it.uniform-analytic1}
The Green operator $G_{B_j(z,s)}$ of the part form of $(\form^{(j)},\domain_j)$ on $B_j(z,s)$ obeys
\begin{equation}\label{e.poi1}
 \|G_{B_j(z,s)}f\|_{L^{\infty}(B_j(z,s),\meas_{j})} \le C_0\Psi_j(s)\|f\|_{L^{\infty}(B_j(z,s),\meas_{j})}.
\end{equation}
\item\label{it.uniform-analytic2} 
Every bounded $\form_{j}$-harmonic function $h$ on $B_j(z,s)$ satisfies
\begin{equation}\label{e.poi2}
 \|h\|_{L^\infty(B_j(z,s/2),\meas_{j})} \le C_0\meas_{j}(B_j(z,s))^{-1/2} \|h\|_{L^2(B_j(z,s),\meas_{j})}.
\end{equation}
\item\label{it.uniform-analytic3} 
Every bounded $\form_{j}$-harmonic function $h$ on $B_j(z,s)$ satisfies
\begin{equation}\label{e.poi3}
 \osc_{B_j(z,\rho)}h\le C_0\left(\frac{\rho}{s}\right)^\eta\|h\|_{L^\infty(B_j(z,s))},\ \text{ for all } \rho\in(0,s/2].
\end{equation}
\end{enumerate}
\end{lemma}
\begin{proof}
By \cite[Theorem~3.3]{EB26}, the assumptions imply a \emph{cutoff Sobolev inequality} with constants depending only on the constants involved in the assumptions. 
\begin{enumerate}[label=\textup{({\roman*})}, align=right, leftmargin=*, topsep=5pt, parsep=0pt, itemsep=2pt]
  \item[\ref{it.uniform-analytic1}]
  This follows from the Faber--Krahn inequality \cite[p.~1497]{GHL15} and \cite[Theorem~9.4]{GHL15}.
  \item[\ref{it.uniform-analytic2}]
  For the local $L^2$-mean-value inequality, observe that $|h|$ is $\form_{j}$-subharmonic whenever $h$ is $\form_{j}$-harmonic, by \cite[Proposition~2.1-(2)]{GHL15}. Thus part~\ref{it.uniform-analytic2} follows from \cite[Theorem~6.3]{GHL15}. 
  \item[\ref{it.uniform-analytic3}]
  Part~\ref{it.uniform-analytic3} follows from \cite[Theorem~3.1 and Corollary~4.2]{BGK12}.
\end{enumerate}
The constants involved in the above conclusions depend only on the constants in $\VD$, $\PI{\Psi_{j}}$, $\Capa{\Psi_{j}}$ and scale constants. Thus the constants involved in the conclusions are independent of $j$.
\end{proof}

\begin{lemma}\label{l.pois}
Under the hypotheses of Lemma~\ref{l.uniform-analytic}, let $R>0$ and $x_j\in \ambient_j$ satisfy $16A_0R<\diam(\ambient_j,\metric_{j})$, and write $B_j=B_j(x_j,4R)$. Suppose $u_j\in\domain_{j,\loc}(B_j)\cap L^\infty(B_j,\meas_{j})$ and $g_j\in L^\infty(B_j,\meas_{j})$ satisfy
\begin{equation}\label{e.Lpoi1}
\form^{(j)}(u_j,\varphi) =-\int_{B_j}g_j\varphi\dif \meas_{j},\ \text{ for all } \varphi\in\domain_j\cap C_c(B_j).
\end{equation}
For each $j$, set
\begin{equation}\label{e.Lpoi2}
 K_j=\meas_{j}(B_j(x_j,R))^{-1/2} \|u_j\|_{L^2(B_j,\meas_{j})} +\Psi_j(R)\|g_j\|_{L^\infty(B_j,\meas_{j})},
 \ \text{ and }\ 
 \theta=\frac{\eta\beta}{\eta+\beta}.
\end{equation}
Then $u_j$ has a unique continuous representative on $B_j$, and, with $C$ independent of $j$,
\begin{equation} \label{e.Lpoi3}
 \sup_{B_j(x_j,2R)}|u_j|\le CK_j, \qquad |u_j(y)-u_j(z)|\le CK_j\left(\frac{\metric_{j}(y,z)}R\right)^\theta, \text{ for all } y,z\in B_j(x_j,R).
\end{equation}
\end{lemma}
\begin{proof}
For the $L^\infty$-estimate, let
\begin{equation}\label{e.Lpoi3+}
	 v_j=G_{B_j(x_j,3R)}(\restr{g_j}{B_j(x_j,3R)}), \ \text{ and }\  h_j=u_j+v_j.
\end{equation}
Then $h_j$ is $\form_{j}$-harmonic on $B_j(x_j,3R)$. Lemma~\ref{l.uniform-analytic}-\ref{it.uniform-analytic1} gives \[\|v_j\|_{L^{\infty}(B_j(x_j,3R),\meas_{j})}\overset{\eqref{e.poi1}}{\le} C\Psi_j(R)\|g_j\|_{L^{\infty}(B_j(x_j,3R),\meas_{j})}.\] For $y\in B_j(x_j,2R)$ apply Lemma~\ref{l.uniform-analytic}-\ref{it.uniform-analytic2} to $h_j$ in $B_j(y,R)\subset B_j(x_j,3R)$. Uniform $\VD$ gives $\meas_{j}(B_j(y,R))\ge c\,\meas_{j}(B_j(x_j,R))$. The triangle inequality yields
\begin{align}
	\|u_j\|_{L^\infty(B_j(x_j,2R),\meas_{j})}&\le \|h_j\|_{L^\infty(B_j(x_j,2R),\meas_{j})}+\|v_j\|_{L^\infty(B_j(x_j,2R),\meas_{j})}\\
	&\overset{\eqref{e.poi2}}{\leq }C \meas_{j}(B_{j})^{-1/2}\|u_j\|_{L^{2}(B_j,\meas_{j})}+\|v_j\|_{L^\infty(B_j(x_j,2R),\meas_{j})}\\
	&\overset{\eqref{e.Lpoi3+}}{\leq} CK_j.\label{e.Lpoi4}
\end{align}

Next, following \cite[Lemma~6.7]{GHH23}, fix $y\in B_j(x_j,R)$ and $0<\rho\le r\le R$. Set \[ v_{j,r}=G_{B_j(y,r)}(g_j|_{B_j(y,r)}),\qquad h_{j,r}=u_j+v_{j,r}.\] The function $h_{j,r}$ is $\form_{j}$-harmonic on $B_j(y,r)$. By Lemma~\ref{l.uniform-analytic}-\ref{it.uniform-analytic1}, \ref{it.uniform-analytic3} and the $L^\infty$-estimate in \eqref{e.Lpoi4},
\begin{equation}
\osc_{B_j(y,\rho)}u_j\le C K_j\biggl(\left(\frac\rho r\right)^\eta +\frac{\Psi_j(r)}{\Psi_j(R)}\biggr) \le C K_j\left(\left(\frac\rho r\right)^\eta +\left(\frac rR\right)^\beta\right).
\end{equation}
For $\rho\le R/2$ choose $r=R(\rho/R)^{\eta/(\eta+\beta)}$; for $\rho>R/2$ use the $L^\infty$ bound \eqref{e.Lpoi4}. The oscillation estimate \eqref{e.poi3} extends to $r/2<\rho\le r$ by the trivial oscillation bound, after enlarging the uniform constant. This gives \[\osc_{B_j(y,\rho)}u_j \le C K_j\left(\frac{\rho}{R}\right)^\theta.\] The oscillation bound gives the stated H\"older estimate and a continuous representative on $B_j(x_j,R)$. Applying the same argument on smaller balls contained in $B_j(x_j,4R)$ and using full support of $\meas_{j}$ yields the unique continuous representative on all of $B_j$. All constants are uniform because the weak scaling bounds for $\Psi_j$ have the same constants, independently of $j$. 
\end{proof}

\subsection{Selection of functions and blow-up scales}

We begin by selecting functions whose energy-measure matrix has full rank on a set of positive measure and which possess sufficient regularity for the subsequent blow-up argument. Although piecewise harmonic approximation is available in this general setting \cite[Lemma~3.8 and Proposition~3.9]{KM20}, we use a dense resolvent core of bounded functions with bounded generators. This choice allows us to exploit the Poisson regularity estimates of Lemma~\ref{l.pois} to obtain quantitative H\"older bounds, particularly well suited to the blow-up procedure. Hino's characterization of martingale dimension \cite[Proposition~2.11 and Theorem~2.13]{Hin13} then provides functions whose energy-measure matrix has full rank on a set of positive measure.

\begin{lemma}\label{l.core}
Let $(\ambient,\metric,\meas,\form,\domain)$ be a MMD space. Let $(-\gen,\Dom(-\gen))$ be its non-negative definite generator. Let $d_{\mathrm{m}}$ be its martingale dimension. Define 
\begin{equation}\label{e.sc19}
\sG:=(I-\gen)^{-1}\bigl(L^2(\ambient,\meas)\cap L^\infty(\ambient,\meas)\bigr).
\end{equation}
Then the following assertions hold.
\begin{enumerate}[label=\textup{(\arabic*)},align=right,leftmargin=*,topsep=5pt,parsep=0pt,itemsep=5pt]
\item\label{it.core1}
The set $\sG$ is a linear subspace dense in $(\domain,\form_1^{1/2})$, and
\begin{equation}\label{e.sc20}
\sG\subset\Dom(\gen)\cap L^\infty(\ambient,\meas),\qquad \gen\sG\subset L^\infty(\ambient,\meas).
\end{equation}
\item\label{it.core2}
If $(\ambient,\metric,\meas,\form,\domain)$ satisfies $\VD$, $\PI{\Psi}$ and $\Capa{\Psi}$, then every element of $\sG$ has a unique continuous representative, which is bounded and locally H\"older continuous.
\item\label{it.core3}
For every $n\in\bN$ with $n\leq d_{\mathrm{m}}$, there exist $u_1,\ldots,u_n\in\sG$ such that, if we define
\begin{equation}\label{e.sc21}
\nu:=\sum_{i=1}^n\Gamma(u_i),\ \text{ and }\ A(x):=\left(\frac{\dif\Gamma(u_i,u_j)}{\dif\nu}(x)\right)_{i,j=1}^n,\ x\in\ambient,
\end{equation}
we have $0<\nu(\ambient)<\infty$, for $\nu$-a.e. $x\in\ambient$, $A(x)$ is symmetric and positive semidefinite, and $\operatorname{tr}A(x)=1$, and
\begin{equation}\label{e.sc22}
\nu\bigl(\Sett{x\in\ambient}{\det A(x)>0}\bigr)>0.
\end{equation}
In particular, $A\in L^1(\ambient,\nu;\bR^{n\times n})$.
\end{enumerate}
\end{lemma}

\begin{proof}
\begin{enumerate}[label=\textup{({\arabic*})}, align=right, leftmargin=*, topsep=5pt, parsep=0pt, itemsep=2pt]
\item[\ref{it.core1}]
Let $\{R_{\alpha}\}_{\alpha\in(0,\infty)}$ be the resolvent of $(\form,\domain)$ on $L^{2}(\ambient,\meas)$, so that  $R_1=(I-\gen)^{-1}$. The resolvent identity and its Markov property \cite[(1.3.7), Corollary~1.3.1 and Theorem~1.4.1]{FOT11} give that $R_{1}g\in\Dom(-\gen)\subset\domain$ for all $g\in L^{2}(\ambient,\meas)$, $\norm{R_{1}g}_{L^{\infty}(\ambient,\meas)}\leq \norm{g}_{L^{\infty}(\ambient,\meas)}$ for all $g\in L^2\cap L^\infty$. By spectral calculus, $\gen R_1g=R_1g-g$, so $\|\gen R_1g\|_\infty\leq2\|g\|_\infty$ for all $g\in L^2\cap L^\infty$. Thus $\sG$ is linear and satisfies \eqref{e.sc20}. If $v\in\domain$ satisfies $v\perp\sG$ in the Hilbert space $(\domain,\form_{1})$, that is 
\begin{equation}
	0=\form_1(v,R_1g)=\langle v,g\rangle_{L^2(\ambient,\meas)},\ \text{ for every }g\in L^2(\ambient,\meas)\cap L^\infty(\ambient,\meas).
\end{equation}
Then by the density of $L^2(\ambient,\meas)\cap L^\infty(\ambient,\meas)$ in $L^2(\ambient,\meas)$, we must have $v=0$. So $\sG$ is dense in $(\domain,\form_{1})$.
\item[\ref{it.core2}] 
Let $u=(I-\gen)^{-1}g\in\sG$, where $g\in L^2(\ambient,\meas)\cap L^\infty(\ambient,\meas)$. By \ref{it.core1}, $u\in\domain$ and $u,\gen u=u-g\in L^\infty(\ambient,\meas)$. For $x\in\ambient$ and a sufficiently small $R>0$, set $B=B(x,4R)$. Then
\begin{equation}
\form(u,\varphi)=-\int_B(u-g)\varphi\dif\meas,\ \text{ for all }\ \varphi\in\domain\cap C_c(B).
\end{equation}
Lemma~\ref{l.pois} gives a continuous representative of $u$ on $B$ which is H\"older continuous on $B(x,R)$. Moreover, by the Markov property of the resolvent,
\[\|u\|_{L^\infty(\ambient,\meas)}\le \|g\|_{L^\infty(\ambient,\meas)}<\infty.\] 
\item[\ref{it.core3}]
Let $\sigma$ be a finite minimal energy-dominant measure, whose existence follows from \cite[Proposition~2.6]{Hin13}. Since $\sG$ is $\form_1$-dense, Hino's rank characterization \cite[Proposition~2.11 and Theorem~2.13]{Hin13}, applied to a countable dense subset of $\sG$, supplies $u_1,\ldots,u_n\in\sG$ such that
\[Z=\left(\frac{\dif \Gamma\langle u_i,u_k\rangle}{\dif \sigma}\right)_{i,k=1}^n\]
is positive definite on a set $S$ of positive $\sigma$-measure. This follows by choosing a full rank minor of size $n$ from the positive semidefinite matrix of a sufficiently large finite family. Set $\nu=\sum_{i=1}^{n}\Gamma(u_i)$ and $q=\operatorname{tr}Z$. Then $\dif\nu=q\dif \sigma$ and $0<\nu(S)\le\nu(\ambient)<\infty$. Moreover,
\[A=\left(\frac{\dif\Gamma\langle u_i,u_k\rangle}{\dif\nu}\right)_{i,k=1}^n=\frac{Z}{q}\qquad \nu\text{-a.e.}\]
The matrix $Z$ is symmetric positive semidefinite because it is the density of a matrix of energy measures; thus $A$ is also symmetric positive semidefinite, with $\operatorname{tr}A=1$ $\nu$-a.e. and $\det A>0$ on $S$. Finally, $|A_{ik}|\le (A_{ii}+A_{kk})/2$, so the entrywise $L^1$-norm satisfies
\[\sum_{i,k=1}^n\|A_{ik}\|_{L^1(\ambient,\nu)}\le n\int_\ambient\operatorname{tr}A\dif \nu=n\nu(\ambient)<\infty.\]
This completes the proof of \ref{it.core3}.
\end{enumerate}
\end{proof}

The next proposition constructs approximate harmonic coordinates at suitably chosen blow-up scales. Starting from functions whose energy-measure matrix has rank at least $n$, we select a point and a sequence of shrinking scales such that, after rescaling and a suitable linear transformation, the energy-measure matrix becomes asymptotically scalar. Lemma~\ref{l.diffe}, based on Hyt\"onen's differentiation theory \cite{Hyt10}, provides the required scales and the control of the reference measure relative to the energy measure. The resulting coordinate maps are uniformly bounded and H\"older continuous on appropriate balls, and their failure to be harmonic vanishes under rescaling.

This construction is inspired by Hino's argument, but avoids the use of compactness of harmonic functions, which is not available in our general setting. Instead of passing to limiting harmonic coordinates with an exactly constant energy-measure matrix, we work with approximate coordinates at finite scales, with errors vanishing in total variation. These properties will be crucial in the subsequent analysis of Euclidean image measures.

\begin{proposition}\label{p.finite-normalization}
Assume the hypotheses of Theorem~\ref{t.main}. Fix an integer $1\le n\le d_{\mathrm m}$ and set $R_*=8(1+\Lambda_{\mathrm{PI}})$. There exist $p\in \ambient$, $r_j\downarrow0$, and the rescaled metric measure Dirichlet spaces defined by 
\[\metric_{j}=r_j^{-1}d, \qquad b_j=\meas(B(p,r_j)),\qquad \meas_{j}=b_j^{-1}\meas,\qquad\form^{(j)}=\frac{\Psi(r_j)}{b_j}\form,\] 
with scale functions 
\[\Psi_j(t)=\frac{\Psi(r_jt)}{\Psi(r_j)},\ \ t\in[0,\infty),\]
and energy measures $\Gamma^{(j)}=\Psi(r_j)b_j^{-1}\Gamma$. There exist continuous maps $\mathbf H_j=(H_j^1,\ldots,H_j^n):\ambient\to\mathbb R^n$, with $H_j^i\in\domain_{\loc}(\ambient)$, normalized by
\[\fint_{B_j(p,4R_*)}H_j^i\dif \meas_{j}=0,\qquad 1\le i\le n, j \in \mathbb{N}\]
and satisfying the following properties.
\begin{enumerate}[label=\textup{(\arabic*)},leftmargin=*]
\item\label{it.finite1}
For every $j\in\mathbb N$, the rescaled metric measure Dirichlet space $(\ambient,\metric_j,\meas_{j},\form^{(j)},\domain)$ satisfies $\VD$, $\PI{\Psi_{j}}$ and $\Capa{\Psi_{j}}$, with structural constants independent of $j$.

\item\label{it.finite2}
On $U_j=B_j(p,R_*)$, the functions $H_j^1,\ldots,H_j^n$ are uniformly bounded and are uniformly H\"older continuous on $B_j(p,R_*/2)$; that is, there exist constants $C<\infty$ and $\theta>0$  such that for every $1\le i\le n$ and every $j \in \mathbb{N}$,
\begin{equation}\label{e.finite-holder}
\sup_{x\in U_j}|H_j^i(x)|\le C,\qquad |H_j^i(x)-H_j^i(y)|\le C \metric_{j}(x,y)^\theta,\quad x,y\in B_j(p,R_*/2).
\end{equation}

\item\label{it.finite3} 
For every $1\le i\le n$, the functions $\{H_{j}^{i}\}_{j\in\bN}$ are \emph{asymptotically harmonic} on $U_j$, in the sense that
\begin{equation}\label{e.finite-poisson}
\form^{(j)}(H_j^i,\varphi)=-\int_\ambient g_j^i\varphi\dif \meas_{j},\qquad\lim_{j \to \infty}\max_{1\le i\le n}\|g_j^i\|_{L^\infty(\ambient,\meas)}=0,
\end{equation}
for every $\varphi\in\domain\cap C_c(U_j)$.

\item\label{it.finite4} 
Moreover, there exist finite positive measures $\nu_j$ such that the energy-measure matrices are \emph{asymptotically scalar}:
\begin{equation}\label{e.finite-matrix}
\Gamma^{(j)}\langle H_j^i,H_j^k\rangle=\delta_{ik}\nu_j+R_j^{ik}\quad\text{on }U_j,\quad  \text{ for } 1\le i,k\le n\ \text{ and }\ j \in \mathbb{N},
\end{equation}
where $R_j^{ik}$ are signed measures on $U_{j}$ satisfying
\[\lim_{j \to \infty}\sum_{i,k=1}^n\|R_j^{ik}\|_{\mathrm{TV}(U_j)}=0.\]
The measures $\nu_j$ also satisfy
\begin{equation}\label{e.finite-masses}
\nu_j(B_j(p,1))=1 \  \mbox{for each $j \in \mathbb{N}$, and}\ \sup_j\bigl(\nu_j(B_j(p,R_*))+\meas_{j}(B_j(p,R_*))\bigr)<\infty.
\end{equation}
\end{enumerate}
\end{proposition}
\begin{proof}
By Lemma~\ref{l.core}-\ref{it.core3}, choose $u_1,\ldots,u_n\in\Dom(\gen)\cap C(\ambient)\cap L^\infty(\ambient,\meas)$ with $\gen u_i\in L^\infty(\ambient,\meas)$ such that, on setting
\begin{equation}\label{e.bl02}
	\nu:=\sum_{i=1}^n\Gamma\langle u_i\rangle ,\qquad A(x):=\left(\frac{\dif\Gamma\langle u_i,u_k\rangle}{\dif\nu}(x)\right)_{1\le i,k\le n},\ x\in\ambient,
\end{equation}
we have $0<\nu(\ambient)<\infty$ and $\nu(\{\det A>0\})>0$.
	
Set $K=5R_*\Lambda_{\mathrm{PI}}$ and let $D>1$ be as in Lemma~\ref{l.diffe}. Applying Lemma~\ref{l.diffe} to $\nu,m,A$ and the prescribed sequence $\rho_j=2^{-j}$, we may choose $p\in\supp\nu$ with $A_0:=A(p)$ positive definite, a constant $C<\infty$, and a subsequence $r_j\downarrow0$ such that
\begin{equation}\label{e.bl-7}
	\nu(B(p,5Kr_j))\le D\nu(B(p,r_j)),\qquad \meas(B(p,Kr_j))\le C\nu(B(p,r_j)),
\end{equation}
and
\begin{equation}\label{e.bl-6}
	\lim_{j\to\infty}\frac{1}{\nu(B(p,r_j))}\int_{B(p,Kr_j)}\|A-A_0\|_1\dif \nu=0.
\end{equation}
Here we have discarded finitely many terms and relabeled the sequence, if necessary.
	
Set $Q:=A_0^{-1/2}$ and $\mathbf u:=(u_1,\ldots,u_n)^{\mathsf T}$. Define
\begin{equation}\label{e.bl9}
b_j:=\meas(B(p,r_j)),\qquad e_j:=\nu(B(p,r_j)),\qquad a_j^2:=\frac{\Psi(r_j)e_j}{b_j},\qquad \nu_j:=e_j^{-1}\nu.
\end{equation}
Since $p\in\supp\nu$, we have $e_j>0$. Let
\[	\mathbf c_j:=\fint_{B(p,4R_*r_j)}\mathbf u\dif \meas,\qquad
  	\mathbf H_j:=a_j^{-1}Q(\mathbf u-\mathbf c_j),\qquad
	\mathbf g_j=\frac{\Psi(r_j)}{a_j}Q\gen\mathbf u,\]
where $\gen\mathbf u=(\gen u_1,\ldots, \gen u_n)^{\mathsf T}$.
	
By regularity and strong locality, $H_j^i\in\domain_{\loc}(\ambient)$. The functions $H_j^i$ are continuous, and their normalization follows directly from the definition of $\mathbf c_j$:
\[\fint_{B_j(p,4R_*)}\mathbf H_j\dif \meas_{j}=a_j^{-1}Q\fint_{B(p,4R_*r_j)}(\mathbf u-\mathbf c_j)\dif \meas=0.\]
	
We verify the four assertions.
	
\begin{enumerate}[label=\textup{(\arabic*)},leftmargin=*]
\item[\ref{it.finite1}]
Since $\metric_{j}=r_j^{-1}d$, $\meas_{j}=b_j^{-1}m$ and $\form^{(j)}=\Psi(r_j)b_j^{-1}\form$, the MMD space $(\ambient,\metric_j,\meas_{j},\form^{(j)},\domain)$ satisfies $\VD$, $\PI{\Psi_{j}}$ and $\Capa{\Psi_{j}}$ with the same structural constants as the original space. The scale functions $\Psi_j$ also satisfy \eqref{e.psibeta} with the same constants.
		
\item[\ref{it.finite2}]
We first establish a uniform local $L^2$-bound for $\mathbf H_j$. By the Poincar\'e inequality $\PI{\Psi_{j}}$, the definition of $\mathbf c_j$ and the boundedness of the fixed matrix $Q$,
\begin{align}
\int_{B_j(p,4R_*)}|\mathbf H_j|^2\dif \meas_{j}
&\le\frac{\|Q\|_{\mathrm{op}}^2}{\Psi(r_j)e_j}\sum_{i=1}^n\int_{B(p,4R_*r_j)}|u_i-c_j^i|^2\dif \meas\\
&\overset{\PI{\Psi_{j}}}{\le}C\frac{\Psi(4R_*r_j)}{\Psi(r_j)e_j}\nu(B(p,4R_*\Lambda_{\mathrm{PI}}r_j))\\
&\overset{\eqref{e.psibeta}}{\le} C_{R_*}\frac{\nu(B(p,Kr_j))}{e_j}\overset{\eqref{e.bl-7}}{\le} C_{R_*}.
		\end{align}
Here we used $4R_*\Lambda_{\mathrm{PI}}<K$, the weak scaling \eqref{e.psibeta} of $\Psi$, and \eqref{e.bl-7}. Consequently,
\begin{equation}\label{e.bl11}
\sup_{j\in\bN}\|\mathbf H_j\|_{L^2(B_j(p,4R_*),\meas_{j};\mathbb R^n)}<\infty.
\end{equation}
		
Next, \eqref{e.bl-7} gives $b_j\le Ce_j$. Therefore,
\[\frac{\Psi(r_j)}{a_j}=\left(\frac{\Psi(r_j)b_j}{e_j}\right)^{1/2}\le C\Psi(r_j)^{1/2}.\]
Since $Q=A_0^{-1/2}$ is fixed and $\gen u_i\in L^\infty(\ambient,\meas)$,
\begin{equation}\label{e.finite-g}
\max_{1\leq i\leq n}\|g_j^i\|_{L^\infty(\ambient,\meas)}\le C\Psi(r_j)^{1/2}\max_{1\leq i\leq n}\|\gen u_i\|_{L^\infty(\ambient,\meas)}\to0,\ \text{ as } j\to\infty.
\end{equation}
		
We apply Lemma~\ref{l.pois} with center $p$, radius $R=R_*/2$, and $u=H_j^i$, $g=g_j^i$. The required radius condition holds for all sufficiently large $j$, since $r_j\downarrow0$. Each $H_j^i$ is bounded for fixed $j$, so the hypotheses of that lemma are satisfied.
		
Since $R_*/2>1$ and $\meas_{j}(B_j(p,1))=1$, we have
\[\meas_{j}(B_j(p,R_*/2))\ge1,\ \text{ and }\ \sup_{j\in\bN}\Psi_j(R_*/2)<\infty.\]
Using these bounds together with \eqref{e.bl11} and \eqref{e.finite-g}, Lemma~\ref{l.pois} gives
\[\sup_{x\in U_j}|H_j^i(x)|\le C,\ \text{ and }\ |H_j^i(x)-H_j^i(y)|\le C\metric_{j}(x,y)^\theta,\quad x,y\in B_j(p,R_*/2),\]
where $C$ and $\theta>0$ are independent of $j$. This establishes \eqref{e.finite-holder}.
		
\item[\ref{it.finite3}]
Since $\mathbf c_j$ is constant and the form is strongly local, \eqref{e.generator-def} gives, for every $\varphi\in\domain\cap C_c(U_j)$,
\begin{equation}\label{e.bl12}
\form^{(j)}(H_j^i,\varphi)=-\int_\ambient g_j^i\varphi\dif \meas_{j}.
\end{equation}
Indeed, this follows from $\form(u_k,\varphi)=-\int_\ambient\gen u_k\,\varphi\dif \meas$ and the definitions of $\form^{(j)}$, $\mathbf H_j$ and $\mathbf g_j$.
		
The convergence $\max_i\|g_j^i\|_\infty\to0$ was established in \eqref{e.finite-g}. Thus \eqref{e.finite-poisson} holds.
		
\item[\ref{it.finite4}]
We first verify the mass bounds. By \eqref{e.bl-7} and the definition of $\nu_j$, we have $\nu_j(B_j(p,1))=1$ and $\nu_j(B_j(p,R_*))\le D$. Moreover, $\VD$ gives
\[\meas_{j}(B_j(p,R_*))=\frac{\meas(B(p,R_*r_j))}{\meas(B(p,r_j))}\le C_{R_*}.\]
This proves \eqref{e.finite-masses}.
		
It remains to establish the energy-measure identity. By strong locality and the normalization of the rescaled form,
\[\Gamma^{(j)}\langle H_j^i,H_j^k\rangle=\sum_{a,b=1}^n Q_{ia}Q_{kb}A_{ab}\,\nu_j.\]
Since $QA_0Q^{\mathsf T}=I_n$, we obtain
\[\Gamma^{(j)}\langle H_j^i,H_j^k\rangle=\delta_{ik}\nu_j+R_j^{ik},\]
where
\[R_j^{ik}=\sum_{a,b=1}^n Q_{ia}Q_{kb}(A_{ab}-A_{0,ab})\,\nu_j.\]
Since $U_j\subset B_j(p,K)$, equation~\eqref{e.bl-6} implies
\[\sum_{i,k=1}^n\|R_j^{ik}\|_{\mathrm{TV}(U_j)}\le\frac{C_Q}{e_j}\int_{B(p,Kr_j)}\|A-A_0\|\dif \nu\overset{\eqref{e.bl-6}}{\to}0,\ \text{ as }j\to\infty.\]
This proves \eqref{e.finite-matrix} and completes part~\ref{it.finite4}.
\end{enumerate}
\end{proof}

\section{Analysis of Euclidean image measures}\label{s.image}

We next construct cutoff functions with uniform energy bounds on the rescaled spaces. The cutoffs $\zeta_j$ are used to localize energy measures and control the boundary terms in the integration-by-parts argument of Lemma~\ref{l.euclidean-image}-\ref{it.image2}, while $\chi_j$ are used in the Caccioppoli argument of Lemma~\ref{l.image-poincare}-\ref{it.image-poincare1} to exclude concentration of the limiting reference measure at a point. The construction relies on the uniform capacity upper bounds and volume doubling of the rescaled spaces in Proposition~\ref{p.finite-normalization}-\ref{it.finite1}.

\begin{lemma}\label{l.uniform-cutoff}
Assume the hypotheses and retain the notation of Proposition~\ref{p.finite-normalization}. In particular, let $U_j=B_j(p,R_*)$, where $R_*=8(1+\Lambda_{\mathrm{PI}})$, and recall that the rescaled spaces $(\ambient,\metric_j,\meas_{j},\form^{(j)},\domain)$ satisfy $\VD$, $\PI{\Psi_{j}}$ and $\Capa{\Psi_{j}}$ with constants independent of $j$.

\begin{enumerate}[label=\textup{(\arabic*)},leftmargin=*]

\item\label{it.cutoff1}
There exist $\zeta_j\in\domain_{j,0}(U_j)\cap C_c(U_j)$ such that
\[0\le\zeta_j\le1,\qquad \zeta_j=1\ \text{on }B_j(p,R_*/4),\qquad\supp\zeta_j\subset B_j(p,R_*/2),\]
and
\begin{equation}\label{e.finite-cutoff}
\sup_{j\in\bN}\bigl(\form^{(j)}(\zeta_j)+\|\zeta_j\|_{L^2(\ambient,\meas_{j})}^2\bigr)<\infty.
\end{equation}
		
\item\label{it.cutoff2}
There exist $\chi_j\in\domain_{j,0}(B_j(p,3/2)) \cap C_c(B_j(p,3/2))$ such that
\[0\le\chi_j\le1,\qquad\chi_j=1\ \text{on }B_j(p,1),\qquad\sup_{j\in\bN}\form^{(j)}(\chi_j)<\infty.\]
\end{enumerate}
\end{lemma}

\begin{proof}
\begin{enumerate}[label=\textup{({\arabic*})}, align=right, leftmargin=*, topsep=5pt, parsep=0pt, itemsep=2pt]
\item[\ref{it.cutoff1}]	
By Proposition~\ref{p.finite-normalization}-\ref{it.finite1}, the rescaled MMD spaces $(\ambient,\metric_j,\meas_{j},\form^{(j)},\domain)$ satisfy $\Capa{\Psi_{j}}$ with uniform structural constants. Applying the $\Capa{\Psi_{j}}$ to $B_j(p,R_*/4)$ gives $\zeta_j$ with the required support and plateau properties, and
\[\form^{(j)}(\zeta_j)\le C_{\mathrm{Cap}}\frac{\meas_{j}(B_j(p,R_*/4))}{\Psi_j(R_*/4)}\le C.\]
Here the last inequality follows from $\VD$, $\meas_{j}(B_j(p,1))=1$, and the uniform weak scaling of $\Psi_j$ by \eqref{e.psibeta}. Moreover,
\[\|\zeta_j\|_{L^2(\ambient,\meas_{j})}^2\le \meas_{j}(B_j(p,R_*/2))\le C.\]
This proves part~\ref{it.cutoff1}.

\item[\ref{it.cutoff2}]
Choose a maximal $1/8$-separated set $\{z_{j,1},\ldots,z_{j,N_j}\}$ in $B_j(p,1)$. By $\VD$, $N_j$ is uniformly bounded. The balls $B_j(z_{j,k},1/8)$ cover $B_j(p,1)$. For each $k$, the capacity upper bound $\Capa{\Psi_{j}}$ provides a cutoff function $\varphi_{j,k}\in\domain\cap C_{c}(\ambient)$ equal to one on $B_j(z_{j,k},1/8)$, supported in $B_j(z_{j,k},1/4)$, and satisfying
\[\form^{(j)}(\varphi_{j,k})\le C\frac{\meas_{j}(B_j(z_{j,k},1/8))}{\Psi_j(1/8)}\le C.\]
The constants are independent of $j$ and $k$. Set $\chi_j:=\max_{1\le k\le N_j} \varphi_{j,k}$. Then $\chi_j=1$ on $B_j(p,1)$ and
\[\supp\chi_j\subset B_j(p,5/4)\Subset B_j(p,3/2).\]
By the contraction property of Dirichlet forms \cite[Chapter~I, Section~2, Exercise~2.4]{BH91},
\[\form^{(j)}(\chi_j)\le\sum_{k=1}^{N_j}\form^{(j)}(\varphi_{j,k})\le C.\]
This completes the proof.
\end{enumerate}
\end{proof}

We recall the definition of the Euclidean $(1,2)$-capacity used below, following \cite[Section~3]{Hin13}. For an open set $U\subset\mathbb R^n$, let
\[\Capacity^{1,2}(U):=\inf\Sett{\int_{\mathbb R^n}\bigl(|u|^2+|\nabla u|^2\bigr)\dif \sL^{n}}{u\in H^1(\mathbb R^n),\ u\ge 1\ \text{a.e.\ on }U}.\]
For an arbitrary set $A\subset\mathbb R^n$, define
\[\Capacity^{1,2}(A):=\inf\left\{\Capacity^{1,2}(U):A\subset U,\ U\text{ open}\right\}.\]
We collect the following Euclidean capacity results from Hino's work \cite[Section~3]{Hin13}.

\begin{proposition}[\!\!{\cite{Hin13}}]\label{p.CapE}  
Let $n\in[2,\infty)\cap\mathbb N$. Let $\rho:\mathbb R^n\to[0,\infty)$ be such that $\rho^{1/2}\in H^1(\mathbb R^n)$. Let $\sigma$ be a Radon measure on $\bR^{n}$ such that $0<\sigma(\bR^{n})<\infty$. 
\begin{enumerate}[label=\textup{({\arabic*})}, align=right, leftmargin=*, topsep=5pt, parsep=0pt, itemsep=2pt]

\item\label{it.Cap1} 
For every compact set $K\subset\bR^{n}$ with $\Capacity^{1,2}(K)=0$, and for every $\epsilon\in(0,\infty)$, there exists $v\in C_{c}^{\infty}(\bR^{n})$ such that $0\leq v\leq1$, $v=1$ on an open neighbourhood of $K$, and
\begin{equation}\label{e.Cap1}
\int_{\bR^{n}}(\abs{\nabla v}^{2}+v^{2})(1+\rho)\dif\sL^{n}<\epsilon.
\end{equation}

\item\label{it.Cap2} 
There exists a $G_{\delta}$ set $N\subset \bR^{n}$ with $\Capacity^{1,2}(N)=0$ such that 
\begin{equation}\label{e.Cap2}
	\sup_{r\in(0,\infty)}r^{-n}\int_{B(y,r)}\rho(z)\dif\sL^{n}(z)<\infty,\ \text{ for all }y\in \bR^{n}\setminus N,
\end{equation}
and 
\begin{equation}\label{e.Cap3}
	\limsup_{r\downarrow 0}r^{2-n}\sigma(B(y,r))=0,\ \text{ for all }y\in \bR^{n}\setminus N.
\end{equation}
		
\item\label{it.Cap3} 
There exists a Borel set $N_{1}\subset\bR^{n}$ with $\sigma(N_{1})=0$ and 
\begin{equation}\label{e.Cap4}
	\liminf_{r\downarrow0}r^{-n}\sigma(B(y,r))>0,\ \text{ for all }y\in \bR^{n}\setminus N_{1}.
\end{equation}
\end{enumerate}
\end{proposition}

\begin{proof}
\begin{enumerate}[label=\textup{({\arabic*})}, align=right, leftmargin=*, topsep=5pt, parsep=0pt, itemsep=2pt] 
\item[\ref{it.Cap1}] 
This is a direct consequence of \cite[Proposition 3.10]{Hin13}

\item[\ref{it.Cap2}] 
\eqref{e.Cap2} is proved in \cite[Lemma 3.9]{Hin13}. \eqref{e.Cap3} is proved in \cite[Lemma 3.12]{Hin13}

\item[\ref{it.Cap3}] 
This is proved in \cite[Lemma 3.11]{Hin13}.
\end{enumerate}
\end{proof}

The following lemma extends the Cauchy--Schwarz inequality for energy measures \cite[(2.4)]{Hin13} to finite families of functions with bounded Borel coefficients. This matrix-valued version will be used in the integration-by-parts argument for Euclidean image measures.
\begin{lemma}[Matrix Cauchy--Schwarz for energy measures]\label{l.matrix-CS}
Let $(\form,\domain)$ be a strongly local regular Dirichlet form on $L^2(\ambient,\meas)$ and let $U\subset \ambient$ be open. For $f_1,\ldots,f_r,g_1,\ldots,g_s\in\domain_{\loc}(U)$, a relatively compact Borel set $E\subset U$, and bounded Borel functions $a_i,b_j$ on $E$, one has
\begin{equation}\label{e.matrix-CS}
	\left|\sum_{i=1}^r\sum_{j=1}^s\int_E a_i b_j\dif\Gamma\langle f_i,g_j\rangle\right|^2
	\le\left(\sum_{i,k=1}^r\int_E a_i a_k\dif \Gamma\langle f_i,f_k\rangle\right)\left(\sum_{j,\ell=1}^s\int_E b_j b_\ell\dif \Gamma\langle g_j,g_\ell\rangle\right).
\end{equation}
The same inequality holds for $E=\ambient$ if all $f_i,g_j\in\domain$ and the coefficients are bounded Borel functions on $\ambient$.
\end{lemma}

\begin{proof}
Write $h=(f_1,\ldots,f_r,g_1,\ldots,g_s)$ and set $\mu=\sum_{k=1}^{r+s}\Gamma\langle h_k\rangle$ on $U$. The mutual energy measures are absolutely continuous with respect to $\mu$, and the matrix
\[Z=\left(\frac{\dif\Gamma\langle h_i,h_j\rangle}{\dif\mu}\right)_{i,j=1}^{r+s}.\]
is positive semidefinite $\mu$-almost everywhere. Indeed, for every rational vector $c$, the measure $\Gamma(\sum_a c_a h_a)$ is nonnegative, and the assertion follows by taking a common null set and using continuity in $c$.
	
Set $v=(a_1,\ldots,a_r,0,\ldots,0)$ and $w=(0,\ldots,0,b_1,\ldots,b_s)$. By positive semidefiniteness and the Cauchy--Schwarz inequality,
	\[\left|\int_E v^{\mathsf T}Zw\dif \mu\right|^2\le\left(\int_E v^{\mathsf T}Zv\dif \mu\right)\left(\int_E w^{\mathsf T}Zw\dif \mu\right).\]
This is precisely \eqref{e.matrix-CS}. All integrals are finite by local finiteness of energy measures and boundedness of the coefficients. The same argument applies to $E=\ambient$ for functions in $\domain$.
\end{proof}

The following lemma concerns the Euclidean pushforward of the normalized energy measures $\nu_j$, weighted by the squares of the cutoff functions $\zeta_j$. More precisely, we study the measures
\[\kappa_j=(\mathbf H_j)_*(\zeta_j^2\nu_j).\]
This is a robust version of \cite[Proposition~3.4]{Hin13}, where Hino considered the pushforward of a weighted energy measure under harmonic functions whose energy-measure matrix is the identity with respect to that measure. In our setting, the coordinate functions are only asymptotically harmonic, and their energy-measure matrices are only approximately scalar, with errors tending to zero in total variation.

The underlying idea of studying the regularity of measures through integration-by-parts identities goes back at least to Fomin's theory of differentiable measures \cite{Fom68} and is closely related to Malliavin calculus \cite{Mal78}; see also \cite{Bog10}. The absolute continuity conclusion is closely related to the energy image density property established in \cite[Theorem~2.65]{EBM25}, which ensures absolute continuity of suitable Euclidean pushforwards of energy measures under a nondegeneracy condition. Here we require the stronger conclusion that the weak limit $\kappa$ of the image measures $\kappa_j$ has a density $\rho$ satisfying $\sqrt\rho\in H^1(\mathbb R^n)$. Following Hino's approach, we obtain this regularity through an integration-by-parts argument. The key difference is that the integration-by-parts identity is obtained only after passing to the limit. A matrix-valued Cauchy--Schwarz estimate controls the derivatives of the approximating measures, while mollification yields the desired regularity of the limiting density.

Throughout the remainder of the paper, weak convergence of finite measures on $\mathbb R^n$ means convergence against bounded continuous functions. That is, for finite Radon measures $\mu_j$ and $\mu$ on $\mathbb R^n$, we write $\mu_j\rightharpoonup\mu$ if
\[ \lim_{j\to\infty}\int_{\mathbb R^n}f\dif \mu_j =\int_{\mathbb R^n}f\dif \mu \qquad\text{for every }f\in C_b(\mathbb R^n).\]
For measures supported in a common compact set, this is equivalent to convergence against all functions in $C_c(\mathbb R^n)$.

\begin{lemma}\label{l.euclidean-image}
Under the setting of Proposition~\ref{p.finite-normalization}, let $\zeta_j$ be the cutoff functions from Lemma~\ref{l.uniform-cutoff}-\ref{it.cutoff1}, satisfying
\[\zeta_j:=1\ \text{on }B_j(p,R_*/4),\qquad \supp\zeta_j\subset B_j(p,R_*/2).\]
Define the following positive measures on $\mathbb R^n$:
\[	\kappa_j:=(\mathbf H_j)_*(\zeta_j^2\nu_j),\qquad
	\eta_j:=(\mathbf H_j)_*\Gamma^{(j)}(\zeta_j),\qquad
	\sigma_{2,j}:=(\mathbf H_j)_*(\zeta_j^2\meas_{j}).\]
Then, after passing to a subsequence, the following assertions hold.
\begin{enumerate}[label=\textup{(\arabic*)},leftmargin=*]
\item\label{it.image1}
The measures $\kappa_j$, $\eta_j$ and $\sigma_{2,j}$ converge weakly to finite positive measures $\kappa$, $\eta$ and $\sigma_2$, respectively. All these measures have support in a common compact subset of $\mathbb R^n$.
		
\item\label{it.image2}
There exists a nonnegative function $\rho$ on $\mathbb R^n$ such that
\begin{equation}\label{e.image-H1}
\kappa=\rho\cdot\mathcal L^n,\qquad\sqrt{\rho}\in H^1(\mathbb R^n),\ \ \text{ and }\ \int_{\mathbb R^n}|\nabla\sqrt{\rho}|^2\dif \sL^{n}\le\liminf_{j\to\infty}\form^{(j)}(\zeta_j).
\end{equation}
\end{enumerate}
\end{lemma}

\begin{proof}
Set $E_*:=\liminf_{j\to\infty}\form^{(j)}(\zeta_j)$.	Passing to a subsequence, we may assume that $\form^{(j)}(\zeta_j)\to E_*$.
	
\begin{enumerate}[label=\textup{(\arabic*)},leftmargin=*]
		
\item[\ref{it.image1}]
By Proposition~\ref{p.finite-normalization}-\ref{it.finite4} and Lemma~\ref{l.uniform-cutoff}-\ref{it.cutoff1}, 
\[	\kappa_j(\mathbb R^n)\le\nu_j(U_j),\qquad
	\eta_j(\mathbb R^n)=\form^{(j)}(\zeta_j),\qquad
	\sigma_{2,j}(\mathbb R^n)\le \meas_{j}(U_j).\]
Thus these measures have uniformly bounded masses. Since $\mathbf H_j$ are uniformly bounded on $U_j$ and $\zeta_j$ have compact support in $U_j$, all three families are supported in a common compact subset of $\mathbb R^n$.
		
By weak compactness of finite measures, we may pass to a common subsequence such that
\[	\kappa_j\rightharpoonup\kappa,\qquad
	\eta_j\rightharpoonup\eta,\qquad
	\sigma_{2,j}\rightharpoonup\sigma_2. \]
This proves part~\ref{it.image1}.
		
\item[\ref{it.image2}]
We first establish an approximate integration-by-parts identity for $\kappa_j$. By Proposition~\ref{p.finite-normalization}-\ref{it.finite4}, we have
\begin{equation}\label{e.image-matrix}
	\Gamma^{(j)}\langle H_j^i,H_j^k\rangle=\delta_{ik}\nu_j+R_j^{ik}\quad\text{on }U_j,
\end{equation}
where
\[\varepsilon_j:=\sum_{i,k=1}^n\|R_j^{ik}\|_{\mathrm{TV}(U_j)}\to0,\ \text{ as }j\to\infty.\]
Moreover,
\begin{equation}\label{e.imagepos}
	\form^{(j)}(H_j^i,\varphi)=-\int_{U_j}g_j^i\varphi\dif \meas_{j},\qquad \max_{1\le i\le n}\|g_j^i\|_{L^{\infty}(\ambient,\meas)}\to0,
\end{equation}
for all $\varphi\in \domain\cap C_{c}(U_{j})$ .
		
For $1\le i,k\le n$ and $j \in \mathbb{N}$, define signed measures on $\mathbb R^n$ by
\[	\alpha_j^i:=2(\mathbf H_j)_*\bigl(\zeta_j\Gamma^{(j)}\langle\zeta_j,H_j^i\rangle\bigr), \qquad 
	\beta_j^{ik}:=(\mathbf H_j)_*(\zeta_j^2R_j^{ik}),\qquad
	\mu_j^i:=(\mathbf H_j)_*(\zeta_j^2g_j^i \meas_{j}).\]
		
Fix $\phi\in C_c^\infty(\mathbb R^n)$. Using $\varphi=\zeta_j^2\phi(\mathbf H_j)$ as a test function in \eqref{e.imagepos}, the Leibniz rule, the chain rule and \eqref{e.image-matrix} give
\begin{equation}\label{e.image-ibp}
\int_{\mathbb R^n}\partial_i\phi\dif \kappa_j =-\int_{\mathbb R^n}\phi\dif \alpha_j^i-\int_{\mathbb R^n}\phi\dif \mu_j^i-\sum_{k=1}^n\int_{\mathbb R^n}\partial_k\phi\dif \beta_j^{ik}.
\end{equation}

Since
\[\max_{1\le i\le n}\|\mu_j^i\|_{\mathrm{TV}}\le\max_{1\le i\le n}\|g_j^i\|_{L^{\infty}(\ambient,\meas)} \meas_{j}(U_j)\to0,\ \sum_{i,k=1}^n\|\beta_j^{ik}\|_{\mathrm{TV}}\le\varepsilon_j\to0, \ \text{ as }j\to\infty,\]
the last two terms in \eqref{e.image-ibp} vanish as $j\to\infty$.
		
We next control $\alpha_j^i$ using Lemma \ref{l.matrix-CS}. For $\boldsymbol\psi=(\psi_1,\ldots,\psi_n)\in C_c^\infty(\mathbb R^n;\mathbb R^n)$, Lemma~\ref{l.matrix-CS} gives
\begin{align}
	\left|\sum_{i=1}^n\int_{\mathbb R^n}\psi_i\dif \alpha_j^i\right|^2
	&\le4\form^{(j)}(\zeta_j)\sum_{i,k=1}^n\int_{U_j}\zeta_j^2\psi_i(\mathbf H_j)\psi_k(\mathbf H_j)\dif \Gamma^{(j)}\langle H_j^i,H_j^k\rangle\\
	&\le4\form^{(j)}(\zeta_j)\left(\int_{\mathbb R^n}|\boldsymbol\psi|^2\dif \kappa_j+\varepsilon_j\|\boldsymbol\psi\|_{\sup}^2\right).
\end{align}
Consequently,
\begin{equation}\label{e.image-alpha-CS}
\left|\sum_{i=1}^n\int_{\mathbb R^n}\psi_i\dif \alpha_j^i\right|^2\le4\form^{(j)}(\zeta_j)\left(\int_{\mathbb R^n}|\boldsymbol\psi|^2\dif \kappa_j+\varepsilon_j\|\boldsymbol\psi\|_{\sup}^2\right).
\end{equation}
		
Taking vector fields with only one nonzero component, we obtain
\[\|\alpha_j^i\|_{\mathrm{TV}}\le2\form^{(j)}(\zeta_j)^{1/2}\bigl(\kappa_j(\mathbb R^n)+\varepsilon_j\bigr)^{1/2}\quad\mbox{for all $j \in \mathbb{N}, 1 \le i \le n$}.\]
Thus the measures $\alpha_j^i$ have uniformly bounded total variations. They are also supported in a common compact subset of $\mathbb R^n$. Passing to a further subsequence, we may assume
\[\alpha_j^i\rightharpoonup\alpha^i,\qquad 1\le i\le n,\]
for some finite signed measures $\alpha^i$. Passing to the limit in \eqref{e.image-ibp} and \eqref{e.image-alpha-CS}, we obtain 
\begin{equation}\label{e.image-limit-ibp}
	\int_{\mathbb R^n}\partial_i\phi\dif \kappa=-\int_{\mathbb R^n}\phi\dif \alpha^i
\end{equation}
and
\begin{equation}\label{e.image-limit-CS}
	\left|\sum_{i=1}^n\int_{\mathbb R^n}\psi_i\dif \alpha^i\right|^2\le 4E_*\int_{\mathbb R^n}|\boldsymbol\psi|^2\dif \kappa.
\end{equation}
		
By the Riesz representation theorem, there exists $\mathbf b=(b_1,\ldots,b_n)\in L^2(\mathbb R^n,\kappa;\mathbb R^n)$ such that
\begin{equation}\label{e.image-b-bound}
	\alpha^i=b_i\kappa,\qquad 1\le i\le n,\ \text{ and }\ \int_{\mathbb R^n}|\mathbf b|^2\dif \kappa\le4E_*.
\end{equation}
Combining this with \eqref{e.image-limit-ibp}, we conclude that
\begin{equation}\label{e.image-measure-derivative}
	\nabla\kappa=\mathbf b\,\kappa\quad\text{in }\mathcal D'(\mathbb R^n).
\end{equation}
It remains to deduce the asserted Sobolev regularity. Let $\epsilon\in(0,\infty)$, $\chi_\epsilon$ be a standard Gaussian mollifier and set $\rho_\epsilon:=\chi_\epsilon*\kappa$. By \eqref{e.image-measure-derivative}, $\nabla\rho_\epsilon=\chi_\epsilon*(\mathbf b\,\kappa)$. The Cauchy--Schwarz inequality gives
\[|\nabla\rho_\epsilon(y)|^2\le\rho_\epsilon(y)\int_{\mathbb R^n}\chi_\epsilon(y-z)|\mathbf b(z)|^2\dif \kappa(z).\]
Therefore,
\[|\nabla\sqrt{\rho_\epsilon}(y)|^2=\frac{|\nabla\rho_\epsilon(y)|^2}{4\rho_\epsilon(y)}\le\frac{1}{4}\int_{\mathbb R^n}\chi_\epsilon(y-z)|\mathbf b(z)|^2\dif \kappa(z).\]
Integrating and using \eqref{e.image-b-bound}, we obtain
\begin{equation}\label{e.image-mollified}
\int_{\mathbb R^n}|\nabla\sqrt{\rho_\epsilon}|^2\dif \sL^{n}\le\frac14\int_{\mathbb R^n}|\mathbf b|^2\dif \kappa\le E_*.
\end{equation}
Here the argument is immediate if $\kappa=0$; otherwise $\rho_\epsilon>0$ everywhere.
		
Since $\int_{\mathbb R^n}\rho_\epsilon\dif \sL^{n}=\kappa(\mathbb R^n)$, the functions $\{\sqrt{\rho_\epsilon}\}_{\epsilon\in(0,\infty)}$ are uniformly bounded in $H^1(\mathbb R^n)$. By weak compactness in $H^1$ and local Rellich compactness, there exist $\epsilon_\ell\downarrow0$ and a nonnegative $h\in H^1(\mathbb R^n)$ such that
\[\sqrt{\rho_{\epsilon_\ell}}\rightharpoonup h		\quad\text{weakly in }H^1(\mathbb R^n),\ \text{ and }\sqrt{\rho_{\epsilon_\ell}}	\to h\quad\text{strongly in }L^2_{\loc}(\mathbb R^n,\sL^{n}), \text{ as }\ell\to\infty.\]

Consequently, $\rho_{\epsilon_\ell}\to h^2$ strongly in $L^1_{\loc}(\mathbb R^n)$. On the other hand, the approximation-of-identity property gives $\rho_{\epsilon_\ell}\mathcal L^n \rightharpoonup\kappa$. Hence $\kappa=h^2\,\mathcal L^n$. Setting $\rho=h^2$, we obtain $\sqrt{\rho}=h\in H^1(\mathbb R^n)$.
		
Finally, weak lower semicontinuity and \eqref{e.image-mollified} yield
\[\int_{\mathbb R^n}|\nabla\sqrt{\rho}|^2\dif y\le E_*=\liminf_{j\to\infty}\form^{(j)}(\zeta_j).\]
This proves \eqref{e.image-H1} and completes	part~\ref{it.image2}.\qedhere
\end{enumerate}
\end{proof}

 We now turn from the energy measures to the reference measures $\meas_{j}$. Specifically, we consider the Euclidean pushforward measures
 \[\sigma_j=(\mathbf H_j)_*\bigl(\one_{B_j(p,2)}\meas_{j}\bigr)\]
 and establish a non-concentration property for their weak limit $\sigma$.
 
In \cite[Proposition~3.13]{Hin13}, Hino proved that the pushforward of the reference measure under a suitable harmonic map cannot be concentrated on a set of zero $H^1$-capacity. His proof uses closability of the Dirichlet form and a comparison between weighted and unweighted Euclidean capacities, an approach related to the theory of closability and fine supports of measures developed by Fukushima, Sato and Taniguchi \cite{FST91}.
 
Our argument establishes the stronger conclusion that $\sigma$ does not charge any set of zero $H^1$-capacity when $n\ge2$. Rather than appealing to closability, we derive a Poincar\'e inequality for $\sigma$ directly from the source  Poincar\'e  inequality and the chain rule. Combined with the Sobolev regularity of the limiting energy-measure density $\rho$ established in the preceding lemma, this yields the desired non-concentration property. An additional ingredient is that $\sigma$ is not a point mass, which follows from approximate harmonicity and the normalization of the coordinate energy measures through a Caccioppoli argument.
 
As in Hino's approach, these properties allow us to select a point satisfying the Euclidean measure-density estimates needed in the proof of Theorem~\ref{t.main}.
 
\begin{lemma}\label{l.image-poincare}
Assume the setting and conclusions of Lemma~\ref{l.euclidean-image}, including the asymptotic harmonicity and energy-measure matrix assumptions. In particular, the rescaled MMD spaces $(\ambient,\metric_j,\meas_{j},\form^{(j)},\domain)$ satisfy $\VD$ and $\PI{\Psi_{j}}$ with constants independent of $j$, by Proposition~\ref{p.finite-normalization}-\ref{it.finite1}. Suppose also that
 	\[\meas_{j}(B_j(p,1))=\nu_j(B_j(p,1))=1\]
 and that the maps $\mathbf H_j$ satisfy uniform local H\"older estimates on $B_j(p,2)$. Assume that the cutoffs $\zeta_j$ equal one $\form_{j}$-q.e. on $B_j(p,2\Lambda_{\mathrm{PI}})$.
 	
Define
\[\sigma_j:=(\mathbf H_j)_*(\one_{B_j(p,2)}\meas_{j}).\]
After passing to a subsequence, $\sigma_j\rightharpoonup\sigma$ weakly as finite measures on $\mathbb R^n$, with the following properties:
 \begin{enumerate}[label=\textup{(\arabic*)},leftmargin=*,itemsep=3pt]
 \item\label{it.image-poincare1}
 $\sigma$ is nonzero and is not a point mass.
 		
 \item\label{it.image-poincare2}
 If $n\ge2$, then every Borel set $N\subset\mathbb R^n$ with $\Capacity^{1,2}(N)=0$ satisfies $\sigma(N)=0$.
 \end{enumerate}
 \end{lemma}
 
\begin{proof}
First we show the subsequential weak convergence of $\sigma_j$ and a Poincar\'e inequality for the weak limit $\sigma$. 
	
By the normalization and uniform volume doubling in Proposition \ref{p.finite-normalization}-\ref{it.finite1}, there exists $C \in (1,\infty)$ such that
\[1=\meas_{j}(B_j(p,1))\le \meas_{j}(B_j(p,2))\le C, \quad \mbox{for all $j \in \mathbb{N}$.}\]
Consequently,
\begin{equation}\label{e.image-sigma-mass}
	1\le \sigma_j(\mathbb R^n)\le C\quad \mbox{for all $j \in \mathbb{N}$.}
\end{equation}
	
Since the maps $\mathbf H_j$ are uniformly bounded on $B_j(p,2)$ by Proposition~\ref{p.finite-normalization}-\ref{it.finite2}, the measures $\sigma_j$ are supported in a common compact subset of $\mathbb R^n$. Hence, after passing to a subsequence, $\sigma_j\rightharpoonup\sigma$ weakly as finite measures on $\mathbb R^n$. Let $\kappa_j=(\mathbf H_j)_*(\zeta_j^2\nu_j)$ be as defined in Lemma~\ref{l.euclidean-image}. By choosing a further subsequence using Lemma~\ref{l.euclidean-image}-\ref{it.image2}, we may assume that there exists $\rho: \mathbf{R}^n \to [0,\infty)$ with $\sqrt{\rho} \in H^1(\mathbb{R}^n)$ such that
\[\kappa_j\rightharpoonup \kappa:=\rho\,\mathcal L^n.\]
	
We first derive a \emph{image Poincar\'e inequality} directly from the Poincar\'e inequality $\PI{\Psi_{j}}$ and the chain rule. Set $V_j=B_j(p,2)$ and $W_j=B_j(p,2\Lambda_{\mathrm{PI}})$. For $u\in C_c^\infty(\mathbb R^n)$ put $v_j=u(\mathbf H_j)$. Since $\sup_{j\in\bN}\Psi_j(2)<\infty$, the scale-invariant Poincaré\'e inequality $\PI{\Psi_{j}}$ with constants independent of $j$ gives
\begin{align}
\inf_{c\in\mathbb R}\int_{\mathbb R^n}|u-c|^2\dif \sigma_j
&\le C\int_{W_j}\dif\Gamma^{(j)}(v_j)=C\sum_{i,k=1}^n\int_{W_j}\partial_i u(\mathbf H_j)\partial_k u(\mathbf H_j)\dif \Gamma^{(j)}\langle H_j^i,H_j^k\rangle\\
&\overset{\eqref{e.image-matrix}}{=}C\int_{W_j}|\nabla u(\mathbf H_j)|^2\dif \nu_j +C\sum_{i,k=1}^n\int_{W_j}\partial_i u(\mathbf H_j)\partial_k u(\mathbf H_j)\dif R_j^{ik}.\label{e.image-PI-finite}
\end{align}

By assumption, $\zeta_j=1$ $\form_{j}$-q.e. on $W_j$. Since energy measures do not charge sets of zero capacity, this equality holds $\nu_j$-almost everywhere. Consequently,
\begin{equation}
\int_{W_j}|\nabla u(\mathbf H_j)|^2\dif \nu_j \le\int_{U_j}\zeta_j^2|\nabla u(\mathbf H_j)|^2\dif \nu_j =\int_{\mathbb R^n}|\nabla u|^2\dif \kappa_j.
\end{equation}
Moreover, the absolute value of the error term is bounded using
\[\abs{\sum_{i,k=1}^n\int_{W_j}\partial_i u(\mathbf H_j)\partial_k u(\mathbf H_j)\dif R_j^{ik}}	\le \varepsilon_j\|\nabla u\|_{\sup}^2,
	\qquad
	\varepsilon_j:=\sum_{i,k=1}^n\|R_j^{ik}\|_{\mathrm{TV}(U_j)}\to0.\]
We conclude that
\begin{equation}\label{e.image-PI-approx}
\inf_{c\in\mathbb R}\int_{\mathbb R^n}|u-c|^2\dif \sigma_j \le C\left(\int_{\mathbb R^n}|\nabla u|^2\dif \kappa_j+\varepsilon_j\|\nabla u\|_{\sup}^2\right). 
\end{equation}
	
It remains to pass to the limit. Since $\sigma_j(\mathbb R^n)\ge1$, by the weak convergence of $\sigma_j$, together with their uniformly compact supports, we know that, as $j\to\infty$,
\begin{align}
\inf_{c\in\mathbb R}\int_{\mathbb R^n}|u-c|^2\dif \sigma_j 
&=\int_{\mathbb R^n}u^2\dif \sigma_j-\frac{1}{\sigma_j(\mathbb R^n)}\left(\int_{\mathbb R^n}u\dif \sigma_j\right)^2\\
&\to \int_{\mathbb R^n}u^2\dif \sigma-\frac{1}{\sigma(\mathbb R^n)}\left(\int_{\mathbb R^n}u\dif\sigma\right)^2=\inf_{c\in\mathbb R}\int_{\mathbb R^n}|u-c|^2\dif \sigma.
\end{align}
	
Furthermore, since $\kappa_j\rightharpoonup\kappa=\rho\,\mathcal L^n$,
\[\int_{\mathbb R^n}|\nabla u|^2\dif \kappa_j\to\int_{\mathbb R^n}|\nabla u|^2\rho\dif \sL^{n}, \quad \mbox{for all $u \in C_c^\infty(\mathbb{R}^n)$}.\]
Letting $j\to\infty$ in \eqref{e.image-PI-approx}, we obtain the pushforward Poincar\'e inequality
\begin{equation}\label{e.image-PI}
	\inf_{c\in\mathbb R}\int_{\mathbb R^n}|u-c|^2\dif \sigma\le C\int_{\mathbb R^n}|\nabla u|^2\rho\dif \sL^{n}, \quad \mbox{for all }u\in C_c^\infty(\mathbb R^n).
\end{equation}
	
\begin{enumerate}[label=\textup{({\arabic*})}, align=right, leftmargin=*, topsep=5pt, parsep=0pt, itemsep=2pt]
 \item[\ref{it.image-poincare1}] 	
 Since the measures $\sigma_j$ are supported in a common compact subset of $\mathbb R^n$, their weak convergence to $\sigma$ implies $\sigma(\mathbb R^n)=\lim_{j\to\infty}\sigma_j(\mathbb R^n)$. Together with \eqref{e.image-sigma-mass}, this gives $1\le\sigma(\mathbb R^n)\le C$.	In particular, $\sigma$ is nonzero. 	Next, we show that the measure $\sigma$ is not a point mass.	

Suppose, toward a contradiction, that $\sigma=M\delta_a$ for some $a\in\mathbb R^n$ and $M\in(0,\infty)$. Since the measures $\sigma_j$ have uniformly compact support, weak convergence implies
\begin{equation}\label{e.image-point-L2}
\lim_{j \to \infty}\int_{B_j(p,2)} |\mathbf H_j(x)-a|^2\dif \meas_{j}(x) = \lim_{j \to \infty} \int_{\mathbb R^n}|y-a|^2\dif \sigma_j(y)=0.
\end{equation}
We claim that
\begin{equation}\label{e.image-point-uniform}
	\lim_{j \to \infty}	\sup_{x\in B_j(p,3/2)}|\mathbf H_j(x)-a|=0.
\end{equation}
To prove this, recall from Proposition~\ref{p.finite-normalization}-\ref{it.finite2} that  $\mathbf H_j$ satisfy a uniform H\"older estimate of the form
\begin{equation}\label{e.image-holder}
|\mathbf H_j(x)-\mathbf H_j(y)|\le C\metric_{j}(x,y)^{\theta},\ \text{ for $x,y\in B_j(p,2)$ with $\metric_{j}(x,y)\le\frac{1}{4}$},
\end{equation}
where $C$ and $\theta>0$ are independent of $j$. Moreover, $\VD$	and the normalization $\meas_{j}(B_j(p,1))=1$ imply that, for every fixed $r\in(0,1/4)$, there exists $c(r)>0$ such that
\begin{equation}\label{e.image-local-lower-volume}
\inf_{x\in B_j(p,3/2)}\meas_{j}(B_j(x,r))\ge c(r),\ \text{ for all $j$}.
\end{equation}
Indeed, for $x\in B_j(p,3/2)$, $B_j(p,1)\subset B_j(x,5/2).$ Applying $\VD$ in Proposition~\ref{p.finite-normalization}-\ref{it.finite1} a finite number of times, with constants depending on $r$ but not on $j$ or $x$, gives \eqref{e.image-local-lower-volume} since
\[1\le \meas_{j}(B_j(x,5/2))\le C(r)\meas_{j}(B_j(x,r)).\]
	
Suppose that \eqref{e.image-point-uniform} fails. Then, after passing to a subsequence, there exist $\delta>0$ and $x_j\in B_j(p,3/2)$ such that $|\mathbf H_j(x_j)-a|\ge\delta$. Choose $r\in(0,1/4)$ such that $Cr^\theta\le{\delta}/{2}$. For every $y\in B_j(x_j,r)$, \eqref{e.image-holder} gives
\begin{equation}\label{e.image-tria}
	|\mathbf H_j(y)-a|\ge|\mathbf H_j(x_j)-a|-|\mathbf H_j(y)-\mathbf H_j(x_j)|\overset{\eqref{e.image-holder}}{\ge}\frac{\delta}{2}
\end{equation}
Since $B_j(x_j,r)\subset B_j(p,2)$, we obtain
\begin{align}
\int_{B_j(p,2)}|\mathbf H_j-a|^2\dif \meas_{j} 
&\ge\int_{B_j(x_j,r)}|\mathbf H_j-a|^2\dif \meas_{j} \\
&\overset{\eqref{e.image-tria}}{\ge}\frac{\delta^2}{4}\meas_{j}(B_j(x_j,r))\stackrel{\eqref{e.image-local-lower-volume}}{\ge}\frac{\delta^2c(r)}4>0,
\end{align}
contradicting \eqref{e.image-point-L2}. This proves \eqref{e.image-point-uniform}.
	
By Lemma~\ref{l.uniform-cutoff}-\ref{it.cutoff2}, choose $\chi_j\in\domain_{j,0}(B_j(p,3/2))$ satisfying	$0\le\chi_j\le1,\chi_j=1$ on $B_j(p,1)$, and
\begin{equation}\label{e.image-chi-energy}
	\sup_j\form^{(j)}(\chi_j)<\infty.
\end{equation}
	
Fix $i\in\{1,\ldots,n\}$ and set $q_j^i=H_j^i-a_i$. Since $\chi_j\in\domain\cap C_c(B_j(p,3/2))$ and $q_j^i\in\domain_{\loc}(\ambient)\cap C(\ambient)$, the local product property and \cite[Theorem 1.4.2-(ii)]{FOT11} imply that
\[\varphi_j=\chi_j^2q_j^i\in\domain\cap C_c(U_j).\]
Here we use that $q_j^i$ is boundedon a neighborhood of $\supp\chi_j$. We may therefore apply Proposition~\ref{p.finite-normalization}-\ref{it.finite3} to obtain
	\[\form^{(j)}(H_j^i,\varphi_j)=	-\int_{U_j}g_j^i\varphi_j\dif \meas_{j}.\]
Since constants have zero energy measure, $\dif\Gamma^{(j)}\langle H_j^i,q_j^i\rangle=\dif\Gamma^{(j)}\langle H_j^i\rangle$. The product rule gives
	\[\dif\Gamma^{(j)}\langle H_j^i,\chi_j^2q_j^i\rangle=\chi_j^2\dif \Gamma^{(j)}\langle H_j^i\rangle+2\chi_jq_j^i\dif \Gamma^{(j)}\langle H_j^i,\chi_j\rangle.\]
Consequently,
\begin{equation}\label{e.icac-identity}
\int_{U_j}\chi_j^2\dif \Gamma^{(j)}\langle H_j^i\rangle+2\int_{U_j}\chi_jq_j^i\dif \Gamma^{(j)}\langle H_j^i,\chi_j\rangle=-\int_{U_j}g_j^i\chi_j^2q_j^i\dif \meas_{j}.
\end{equation}
	
Define
\[A_j^i=\int_{U_j}\chi_j^2\dif \Gamma^{(j)}\langle H_j^i\rangle, \qquad 	B_j^i=\int_{U_j}(q_j^i)^2\dif \Gamma^{(j)}\langle\chi_j\rangle.\]
By the Cauchy--Schwarz inequality for energy measures,
\begin{equation}\label{e.icac-identity1}
\left|\int_{U_j}\chi_jq_j^i\dif \Gamma^{(j)}\langle H_j^i,\chi_j\rangle\right|\le	(A_j^i)^{1/2}(B_j^i)^{1/2}.
\end{equation}
	
Define for each $j \in \mathbb{N}$,
\[M_j:=\sup_{B_j(p,3/2)}|\mathbf H_j-a|,\qquad G_j:=\max_{1\le i\le n}\|g_j^i\|_{L^\infty(U_j)}.\]
Since $\chi_j$ is supported in $B_j(p,3/2)$, we have $B_j^i\le M_j^2\form^{(j)}(\chi_j)$. Moreover,
\begin{equation}\label{e.icac-identity2}
\left|\int_{U_j}g_j^i\chi_j^2q_j^i\dif \meas_{j}\right|  \le G_jM_j\meas_{j}(B_j(p,3/2)).
\end{equation}		
	
It follows from \eqref{e.icac-identity}, \eqref{e.icac-identity1} and \eqref{e.icac-identity2} that
\begin{align}
A_j^i&\le 2(A_j^i)^{1/2}(B_j^i)^{1/2}+G_jM_j\meas_{j}(B_j(p,3/2)) \\
&\le\frac12A_j^i+2B_j^i+G_jM_j\meas_{j}(B_j(p,3/2)),
\end{align}
where we used the Cauchy-Schwarz inequality. Rearranging gives the Caccioppoli estimate  
\begin{equation}
A_j^i\le4B_j^i+2G_jM_j\meas_{j}(B_j(p,3/2))\le4M_j^2\form^{(j)}(\chi_j)+2G_jM_j\meas_{j}(B_j(p,3/2)). \label{e.icac}
\end{equation}
	
By $\VD$ in Proposition~\ref{p.finite-normalization}-\ref{it.finite1}, we have 
\begin{equation} \label{e.mjunif}
\sup_{j\in\bN}\meas_{j}(B_j(p,3/2))<\infty.
\end{equation}
By \eqref{e.image-point-uniform}, $\lim_{j \to \infty} M_j=0$, while the asymptotic harmonicity in Proposition~\ref{p.finite-normalization}-\ref{it.finite3} gives $\lim_{j \to \infty}G_j=0$. Therefore \eqref{e.icac}, \eqref{e.mjunif} and \eqref{e.image-chi-energy} imply $\lim_{j \to \infty} A_j^i=0$. Since $\chi_j=1$ on $B_j(p,1)$, we have $\Gamma^{(j)}\langle H_j^i\rangle(B_j(p,1)) \le A_j^i$ for all $j$. Thus for each $i$, we have $\Gamma^{(j)}\langle H_j^i\rangle(B_j(p,1))\to0$ as $j\to\infty$. Summing over $i$, we obtain
\begin{equation}\label{e.image-energy-vanishes}
	\lim_{j \to \infty}	\sum_{i=1}^n
		\Gamma^{(j)}\langle H_j^i\rangle(B_j(p,1))
	 = 0.
\end{equation}
	
On the other hand, by Proposition~\ref{p.finite-normalization}-\ref{it.finite4}
\[\lim_{j \to \infty} \sum_{i=1}^n \Gamma^{(j)}\langle H_j^i\rangle(B_j(p,1)) =  \lim_{j \to \infty} \left( n\nu_j(B_j(p,1)) + \sum_{i=1}^n R_j^{ii}(B_j(p,1)) \right)     =n>0,\]
contradicting \eqref{e.image-energy-vanishes}. 	Thus $\sigma$ cannot be a point mass.

\item[\ref{it.image-poincare2}]
It remains to prove that $\sigma$ does not charge sets of zero
		$H^1$ capacity, when $n \ge 2$.	The main idea is to combine the pushforward	Poincar\'e inequality with the Sobolev regularity	of $\sqrt{\rho}$ and the fact that $\sigma$ is not a point mass.
	
Assume $n\ge2$ and let $N\subset\mathbb R^n$ be a Borel set	satisfying	$\Capacity^{1,2}(N)=0$. Suppose, toward a contradiction, that $\sigma(N)>0$. Since $\sigma$ is a finite measure which is not a point mass, there exists a Borel set $A\subset\mathbb R^n$ such that $0<\sigma(A)<\sigma(\mathbb R^n)$. As $\sigma(N)>0$, we have 
	\[\text{either } \sigma(N\cap A)>0,\ \text{ or }\ \sigma(N\cap A^c)>0.\]
	
\begin{itemize}
	\item If $\sigma(N\cap A)>0$, the inner regularity of $\sigma$ allows us to choose compact sets $F\subset N\cap A$ and $G\subset A^c$ with $\sigma(F)>0$ and $\sigma(G)>0$;
	\item If instead $\sigma(N\cap A^c)>0$, we chosoe $F\subset N\cap A^c$ and $G\subset A$ with $\sigma(F)>0$ and $\sigma(G)>0$.
\end{itemize} 
In either case, $F$ and $G$ are disjoint compact sets satisfying $F\subset N$, $\sigma(F)>0$ and $\sigma(G)>0$. Since $\Capacity^{1,2}(F)\le\Capacity^{1,2}(N)=0$. Proposition~\ref{p.CapE}-\ref{it.Cap1} implies that there exist functions $v_k\in C_c^\infty(\mathbb R^n)$ satisfying
\[0\le v_k\le1, \qquad v_k=1\quad\text{on an open neighbourhood of }F,\]
and
\begin{equation}\label{e.image-vk-energy}
	\lim_{k \to \infty} \int_{\mathbb R^n}\left(v_k^2+|\nabla v_k|^2\right)(1+\rho)\dif \sL^{n}= 0.
\end{equation}
	
Since $F$ and $G$ are disjoint compact subsets of $\mathbb R^n$, there exists $\psi\in C_c^\infty(\mathbb R^n)$ 	such that $0\le\psi\le1$, $\psi=1$ near $F$ and $\psi=0$ near $G$. Define $u_k:=\psi v_k$. Then
\begin{equation}\label{e.image-uk-values}
		u_k=1\quad\text{on }F,\ \text{ and }\ u_k=0\quad\text{on }G.
	\end{equation}
By the product rule, $\nabla u_k = \psi\nabla v_k+v_k\nabla\psi$. Consequently, $|\nabla u_k|^2	\le2\psi^2|\nabla v_k|^2+2v_k^2|\nabla\psi|^2$. Since $\psi$ and $\nabla\psi$ are bounded,
\begin{align}
0\leq \int_{\mathbb R^n}|\nabla u_k|^2\rho\dif \sL^{n} &\le2\|\psi\|_{\sup}^2\int_{\mathbb R^n}|\nabla v_k|^2\rho\dif \sL^{n} +2\|\nabla\psi\|_{\sup}^2\int_{\mathbb R^n}v_k^2\rho\dif \sL^{n}\\
&\le C_\psi\int_{\mathbb R^n}\left(|\nabla v_k|^2+v_k^2\right)(1+\rho)\dif\sL^{n}\overset{\eqref{e.image-vk-energy}}{\to}  0,\ \text{ as }k\to\infty.
\label{e.image-uk-energy}
\end{align}

On the other hand, for every $c\in\mathbb R$,
\[\int_{\mathbb R^n}|u_k-c|^2\dif \sigma \overset{\eqref{e.image-uk-values}}{\ge} \int_F|1-c|^2\dif \sigma + \int_G|c|^2\dif \sigma = \sigma(F)(1-c)^2+\sigma(G)c^2.\]
Minimizing the right-hand side over $c\in\bR$ gives
\begin{equation}\label{e.image-variance-lower}
\inf_{c\in\mathbb R}\int_{\mathbb R^n}|u_k-c|^2\dif \sigma\ge\frac{\sigma(F)\sigma(G)}{\sigma(F)+\sigma(G)}>0,
\end{equation}
independently of $k$. However, applying the Poincar\'e inequality \eqref{e.image-PI} to $u_k$ gives for each $k$
\[\inf_{c\in\mathbb R}	\int_{\mathbb R^n}|u_k-c|^2\dif \sigma	\le	C\int_{\mathbb R^n}	|\nabla u_k|^2\rho\dif y.	\]
By letting $k \to \infty$ and using \eqref{e.image-uk-energy} and \eqref{e.image-variance-lower}, we obtain the desired contradiction. Hence $\sigma(N)=0$. \qedhere
\end{enumerate}
\end{proof}

\section{Proof by contradiction using Sobolev inequalities}\label{s.proof}

Recall that a sequence $\{x_{i}\}_{i=0}^{N}$ of points in $\ambient$ is an \emph{$\epsilon$-chain connecting $x,y\in\ambient$}, if 
\begin{equation}
	x_{0}=x,\ x_{N}=y,\ \text{ and }\ \max_{1\leq i\leq N}\metric(x_{i-1},x_{i})<\epsilon.
\end{equation}
The following connectivity is a localized version of \cite[Lemma~2.2]{Mur20}.
\begin{lemma}\label{l.local-chains}
Let $(\ambient,\metric,\meas,\form,\domain)$ be a MMD space satisfying $\PI{\Psi}$, with dilation constant $\Lambda_{\mathrm{PI}}$.	Let $L>\Lambda_{\mathrm{PI}}$. 	For every $p\in \ambient$, every $R>0$, and every $\epsilon>0$, any two points in $B(p,R)$ can be joined by an $\epsilon$-chain contained in $B(p,LR)$.
\end{lemma}

\begin{proof}
Set $U=B(p,LR)$ and fix $\epsilon>0$. Let $C$ be an $\epsilon$-chain component of $U$. Both $C$ and $U\setminus C$ are relatively open in $U$, so $f=\one_C$ is locally constant on $U$. By regularity and strong locality, $f\in\domain_{\loc}(U)$ and $\Gamma\langle f\rangle(U)=0$. Since $B(p,\Lambda_{\mathrm{PI}}R)\subset U$, the Poincar\'e  inequality $\PI{\Psi}$ gives
\[ \int_{B(p,R)}|f-f_{B(p,R)}|^2\dif \meas=0.\]
As $f$ is continuous on $U$ and $\meas$ has full support, $f$ is constant on $B(p,R)$.
	
Taking $C$ to be the component containing $p$, we conclude that $B(p,R)\subset C$. This proves the assertion.
\end{proof}

The proof of the following lemma is a variant of the chaining argument in \cite[Lemma~2.4]{Mur25}. As in that work, we apply the  Poincar\'e   inequality along a chain of balls to control the oscillation of a capacitary cutoff function. The main difference is that we do not assume a chain condition or homogeneous volume growth.

\begin{lemma} \label{l.ds-lower}	
Assume $\VD$, $\PI{\Psi}$ and $\Capa{\Psi}$, with $\Psi$ satisfying \eqref{e.psibeta}. Then there exist $C,r_0>0$ such that
\begin{equation}\label{e.ds-scale-comparison}
	\frac{\Psi(R)}{\Psi(r)}\le C\left(\sup_{x\in \ambient}\frac{V(x,R)}{V(x,r)}\right)^2,\ \text{ for all }0<r<R<r_0.
\end{equation}
Consequently, $d_{\mathrm s}\in[1,2\alpha/\beta]$, where $\alpha$ and $\beta$ are constants from \eqref{e.VD-polynomial} and \eqref{e.psibeta}, respectively. 
\end{lemma}

\begin{proof}
Fix $p\in \ambient$. By \cite[Corollary~2.3]{Mur20}, $\VD$ and $\PI{\Psi}$ imply that $(\ambient,\metric)$ is uniformly perfect. More precisely, whenever $B(x,r)\ne \ambient$, we have $B(x,r)\setminus B(x,r/2)\ne\emptyset$. Since $\#\ambient\ge2$, we may choose $0<R< \diam(\ambient,\metric)/32$ so that $B(p,8R)\neq \ambient$ for all $p \in \ambient$. Consequently, there exists $q\in B(p,8R)\setminus B(p,4R)$.Write
\[M(R,r):=\sup_{x\in \ambient}\frac{V(x,R)}{V(x,r)},\ \text{ for }0<r<R<\infty.\]
Fix $L>9\Lambda_{\mathrm{PI}}$ and put $U=B(p,LR)$. By Lemma~\ref{l.local-chains}, applied to $B(p,9R)$, the points $p$ and $q$ can be joined by an $\epsilon$-chain contained in $U$, for every $\epsilon>0$.
	
Let $r\in(0,10^{-2}(1+\Lambda_{\mathrm{PI}})^{-1}{R})$. and choose a maximal $r$-separated set $Z\subset U$ containing $p$ and $q$. For every $z\in Z$, $\VD$ and the definition of $M(R,r)$ give
\[V(z,r/3)\ge cV(z,r)\ge c_L\frac{V(p,R)}{M(R,r)}.\]
The balls $B(z,r/3)$ are pairwise disjoint and contained in $B(p,(L+1)R)$. Another application of $\VD$ therefore yields $\#Z\le C_LM(R,r)$. Choose an $(r/4)$-chain from $p$ to $q$ inside $U$ and project its vertices onto $Z$. After erasing loops, we obtain distinct points $z_0=p,z_1,\ldots,z_N=q$	such that
\begin{equation} \label{e.bnd-L}
	\max_{1\leq k \leq N}d(z_{k-1},z_k)<3r,\ \text{ and }\ N\le C_LM(R,r).
\end{equation}
By $\Capa{\Psi}$, there exists $f\in\domain\cap C_c(\ambient)$ such that
\begin{equation}\label{e.ds-cutoff-energy}
	0\le f\le1,\  f=1\ \text{on }B(p,R),\ 	f=0\ \text{on }\ambient\setminus B(p,2R),\ \text{ and }	\form(f)\le C\frac{V(p,R)}{\Psi(R)}.
\end{equation}
Set $B_k=B(z_k,r)$ and $a_k=\fint_{B_k}f\dif \meas$. Since $f=1$ on $B(p,R)$ and vanishes outside $B(p,2R)$, we have $a_0=1$ and $a_N=0$. For consecutive vertices, the balls $B_{k-1}$ and $B_k$ lie in $B(z_{k-1},5r)$. The  Poincar\'e  inequality  $\PI{\Psi}$  and $\VD$ imply
\begin{equation}\label{e.ds-edge}
	|a_k-a_{k-1}|^2\le C\frac{\Psi(r)}{V(z_{k-1},r)}\Gamma\langle f\rangle(5\Lambda_{\mathrm{PI}}B_{k-1})\leq C_{1}\frac{\Psi(r)M(R,r)}{V(p,R)}\Gamma\langle f\rangle(5\Lambda_{\mathrm{PI}}B_{k-1}).
\end{equation}
Since the balls $\{B(z_k,5\Lambda_{\mathrm{PI}}r)\}_{k=1}^{N}$ have uniformly bounded overlap, by Cauchy--Schwarz inequality, and summing \eqref{e.ds-edge}, we obtain
\begin{align}
1=|a_N-a_0|^2
&\le N\sum_{k=1}^N|a_k-a_{k-1}|^2\overset{\eqref{e.bnd-L},\eqref{e.ds-edge}}{\leq}C\frac{\Psi(r)M(R,r)^2}{V(p,R)}\sum_{k=1}^{N}\Gamma\langle f\rangle(5\Lambda_{\mathrm{PI}}B_{k-1})\\
&\leq C\frac{\Psi(r)M(R,r)^2}{V(p,R)}\form(f)\overset{\eqref{e.ds-cutoff-energy}}{\leq} C\frac{\Psi(r)}{\Psi(R)}M(R,r)^{2}.\label{e.chain-ener}
\end{align}
This proves \eqref{e.ds-scale-comparison} when $r\in(0,10^{-2}(1+\Lambda_{\mathrm{PI}})^{-1}{R})$. For $r\in [10^{-2}(1+\Lambda_{\mathrm{PI}})^{-1}{R},R)$, the estimate follows from weak scaling of $\Psi$ in \eqref{e.psibeta}, after increasing $C$.
	
Finally, take $r=\lambda R$. Then for every $\lambda\in(0,1)$,
\[\sup_{x\in \ambient}D(x,R,\lambda)\overset{\eqref{e.ds-scale-comparison}}{\ge}1-\frac{\log C}{\log(\Psi(R)/\Psi(\lambda R))}\overset{\eqref{e.psibeta}}{\geq} 1-\frac{\log C}{\beta\log\lambda^{-1}-\log C_{1}}, \text{ for all }R\in(0,\infty). \]
Letting $\lambda\downarrow0$, we obtain $d_{\mathrm s}\ge1$. By \eqref{e.VD-polynomial} and \eqref{e.psibeta}, we have
\[D(x,r,\lambda)\leq \frac{2\alpha\log\lambda^{-1}-\log C_{2}}{\beta\log\lambda^{-1}-\log C_{3} },\ \text{ for all }x\in\ambient,\ r\in(0,\infty),\ \lambda\in(0,1) \]
for some constant $C_{2},C_{3}>0$. Letting $\lambda\downarrow0$, we obtain $d_{\mathrm s}\leq 2\alpha/\beta$.
\end{proof}

The following lemma makes explicit the volume growth information encoded in the uniform small-scale spectral  dimension. For every $s>d_{\mathrm s}$, it provides a lower bound on the relative volumes of concentric balls, uniformly over their centres and at sufficiently small scales. By volume doubling, the estimate extends to balls whose centres lie within a fixed multiple of the reference radius. These estimates will be used to obtain uniform heat-kernel upper bounds for the rescaled Dirichlet forms, and hence the ultracontractivity estimates needed in the final argument.
 
\begin{lemma}\label{l.uniform-profile}
Assume $\VD$, \eqref{e.psibeta}, and let $d_{\mathrm{s}}$ be defined in \eqref{e.ds-unif}. For every $s\in(d_{\mathrm s},\infty)$ there exist $c_s>0$ and $r_s>0$ such that
\begin{equation}\label{e.uniform-profile}
	\frac{\meas(B(x,tr))}{\meas(B(x,r))}\ge c_s\left(\frac{\Psi(tr)}{\Psi(r)}\right)^{s/2},	\ \text{ for all } x\in \ambient,\ r\in(0,r_s),\ t\in(0,1].
\end{equation}
Moreover, for each fixed $R\in[1,\infty)$ there is $c_{s,R}>0$ such that
\begin{equation}\label{e.uniform-nearby}
\frac{\meas(B(x,tr))}{\meas(B(p,r))}\ge c_{s,R}\left(\frac{\Psi(tr)}{\Psi(r)}\right)^{s/2},	\ \text{ for all } x\in B(p,Rr),\ r\in(0,r_s),\ t\in(0,1].
\end{equation}
\end{lemma}

\begin{proof}
Choose $s_0$ with $d_{\mathrm s}<s_0<s$. The definition of the $\liminf$ in \eqref{e.ds-unif} gives a fixed $\lambda_0\in(0,1)$ for which
\[\limsup_{r\downarrow0}\sup_{x\in \ambient}D(x,r,\lambda_0)<s_0.\]
After decreasing $r_s$, we have
\[\frac{V(x,\lambda_0r)}{V(x,r)}\ge\left(\frac{\Psi(\lambda_0r)}{\Psi(r)}\right)^{s_0/2},\ \text{ for every $x\in\ambient$ and $r\in(0,r_s)$}.\]
Since $\lambda_{0}\in(0,1)$, the same inequality applies at the radii $r,\lambda_0r,\ldots,\lambda_0^{k-1}r$ for each $k\in\bN$. By multiplying these inequalities, we obtain
\begin{equation}\label{e.profile-iterate}
	\frac{V(x,\lambda_0^kr)}{V(x,r)}\ge\left(\frac{\Psi(\lambda_0^kr)}{\Psi(r)}\right)^{s_0/2},\ \text{ for all $x\in\ambient$, $r\in(0,r_s)$ and $k\in\bN$.}
\end{equation}
For $t\in(0,1]$, we choose $k\in\bN$ such that $t\in(\lambda_0^{k+1},\lambda_0^k]$. By the monotonicity of measures, \eqref{e.profile-iterate}, and weak scaling of $\Psi$ in \eqref{e.psibeta} give, with a constant $c$ depending only on $\lambda_0,s_0,\Psi$,
\[\frac{V(x,tr)}{V(x,r)}\ge c\left(\frac{\Psi(tr)}{\Psi(r)}\right)^{s_0/2} \ge c\left(\frac{\Psi(tr)}{\Psi(r)}\right)^{s/2}.\]
This is \eqref{e.uniform-profile}. Finally, if $x\in B(p,Rr)$, the inclusion $B(p,r)\subset B(x,(R+1)r)$ and iterated $\VD$ imply $V(x,r)\ge c_RV(p,r)$. Combining this with	\eqref{e.uniform-profile} proves \eqref{e.uniform-nearby}.
\end{proof}

The next lemma translates the small-scale volume growth encoded in $d_{\mathrm s}$ into uniform ultracontractivity estimates for the rescaled Dirichlet forms. We work with semigroups killed on fixed rescaled balls, since the relevant volume estimates are local rather than global. The proof combines the heat-kernel upper bounds implied by our assumptions with the definition of $d_{\mathrm s}$. The crucial point is that the resulting estimates are uniform in the blow-up parameter $j$. In the subsequent lemma, they yield uniform Sobolev inequalities, which provide the final contradiction in the proof of Theorem~\ref{t.main}.

\begin{lemma} \label{l.local-killed}
Assume the hypotheses of Proposition~\ref{p.finite-normalization}. Fix $s\in(d_{\mathrm s},\infty)$ and let $R_*=8(1+\Lambda_{\mathrm{PI}})$ be as in Proposition~\ref{p.finite-normalization}. Set $U_j=B_j(p,R_*)$. Let $(P_t^{U_j})_{t\in(0,\infty)}$ be the semigroup of the part Dirichlet form $(\form^{(j),U_j},\domain_{0}(U_j))$ on $L^2(U_j,\meas_{j})$. There exist $C=C(s,R_*)>0$ and $j_0\in\bN$ such that, for all $j\ge j_0$ and $t\in(0,1]$,
\begin{equation}\label{e.finite-ultra}
	\|P_t^{U_j}\|_{L^1(U_j,\meas_{j})\to L^\infty(U_j,\meas_{j})}\le Ct^{-s/2}.
\end{equation}
\end{lemma}

\begin{proof}
Recall that
\[	\Psi_j(r)=\frac{\Psi(r_jr)}{\Psi(r_j)},\qquad
	\meas_{j}=\frac{\meas}{\meas(B(p,r_j))},\qquad
	\metric_{j}=\frac{\metric}{r_j}. \]
By \cite[Theorem 3.3]{EB26} and \cite[Theorem 1.2]{GHL15}, the rescaled global heat kernel satisfies
\[p_t^{(j)}(x,y)\le\frac{C}{\meas_{j}(B_j(x,\Psi_j^{-1}(t)))}.\]
Note that $\sup_{j\in\bN}\sup_{t\in(0,1]}\Psi_j^{-1}(t)\leq C$. Since $x\in U_j=B_j(p,R_*)$, we have 
\begin{align}
\meas_{j}(B_j(x,\Psi_j^{-1}(t)))=	\frac{\meas(B(x,r_j\Psi_j^{-1}(t)))}{\meas(B(p,r_j))} 	
&\geq \frac{\meas(B(x,r_j C^{-1}\Psi_j^{-1}(t)))}{\meas(B(p,r_j))} \\
&\overset{\eqref{e.uniform-nearby}}{\ge}c_{s,R}\left(\frac{\Psi(r_jC^{-1}\Psi_j^{-1}(t))}{\Psi(r_j)}\right)^{s/2} =c'_{s,R}t^{s/2}.
\end{align}
The constant is independent of $j$, provided $r_j$ is sufficiently small.
	
Since heat kernel of the killed semigroup is dominated by the global heat kernel, that is,
\[0\le p_t^{U_j}(x,y)\le p_t^{(j)}(x,y).\]
we have
\[\|P_t^{U_j}\|_{L^1(U_j,\meas_{j})\to L^\infty(U_j,\meas_{j})}	\le	\esssup_{x,y\in U_j}	p_t^{U_j}(x,y)	\le Ct^{-s/2},	\]
which proves \eqref{e.finite-ultra}.
\end{proof}

The following Sobolev estimates are standard consequences of ultracontractivity, through the Hardy--Littlewood--Sobolev theory for symmetric sub-Markovian semigroups developed by Varopoulos \cite{Var85}; see also \cite[Chapter~II]{VSC92}. For the low-dimensional case, we use the Gagliardo--Nirenberg inequality of Coulhon \cite[Th\'eor\`eme~1]{Cou92}.
 
\begin{lemma}\label{l.truncated-sob}
Fix $s\in(d_{\mathrm s},\infty)$. Let $(\form^{(j),U_j},\domain_{j,0}(U_j))$ be the part Dirichlet forms from Lemma~\ref{l.local-killed}. Write $\form_1^{(j)}(\cdot):=\form^{(j)}(\cdot) +\|\cdot\|_{L^2(\ambient,\meas_{j})}^2$.
	
\begin{enumerate}[label=\textup{(\arabic*)},leftmargin=*]

\item\label{it.sob-infty} 
If $s<2$, then
\begin{equation}\label{e.truncated-infty}
	\|v\|_{L^\infty(U_j,\meas_{j})}	\le C_s\form_1^{(j)}(v)^{1/2},		\quad\text{ for all } v\in\domain_{j,0}(U_j).
\end{equation}
	
\item\label{it.sob-q} 
If $s>2$, then
\begin{equation}\label{e.truncated-q}
	\|v\|_{L^{\frac{2s}{s-2}}(U_j,\meas_{j})}\le C_s\form_1^{(j)}(v)^{1/2},\quad\text{ for all } v\in\domain_{j,0}(U_j).
\end{equation}
\end{enumerate}
In both cases, the constant $C_s$ is independent of $j$, for all sufficiently large $j$.
\end{lemma}

\begin{proof}
Let $\gen_j^{U_j}$ denote the nonpositive self-adjoint generator of the part Dirichlet form corresponding to $\form^{(j)}$ on $U_j$, and set $A_j=I-\gen_j^{U_j}$. Thus $A_j$ is positive self-adjoint and
\[e^{-tA_j}=e^{-t}P_t^{U_j},\qquad t\in(0,\infty).\]
By Lemma~\ref{l.local-killed}, the semigroup property and $L^1$-contractivity,
\[\|P_t^{U_j}\|_{L^1(U_j,\meas_{j})\to L^\infty(U_j,\meas_{j})}
	\le C(t\wedge1)^{-s/2},\qquad t>0.\]
Consequently, 
\begin{equation}\label{e.sob-ultra}
\|e^{-tA_j}\|_{L^1(U_j,\meas_{j})\to L^\infty(U_j,\meas_{j})}		\le C_s t^{-s/2},\qquad t>0,
\end{equation}
with $C_s$ independent of $j$.
\begin{enumerate}[label=\textup{({\arabic*})}, align=right, leftmargin=*, topsep=5pt, parsep=0pt, itemsep=2pt]
\item[\ref{it.sob-infty}]
If $s<2$, Coulhon's Gagliardo--Nirenberg inequality \cite[Th\'eor\`eme~1]{Cou92}, applied with $\alpha=1$ and $p=2$, gives
\[\|v\|_\infty\le C_s\|v\|_2^{1-s/2}\|A_j^{1/2}v\|_2^{s/2}\le C_s\form_1^{(j)}(v)^{1/2}.\]
  
\item[\ref{it.sob-q}]	
If $s>2$, Varopoulos' Sobolev inequality \cite[Theorem, p.~259]{Var85} gives
\[\|v\|_{2s/(s-2)}\le C_s\|A_j^{1/2}v\|_2=C_s\form_1^{(j)}(v)^{1/2}.\]
\end{enumerate}

In both cases, the cited inequalities have constants depending only on $s$ and the ultracontractivity constant in \eqref{e.sob-ultra}, and are therefore uniform in $j$.
\end{proof}

\begin{proof}[Proof of Theorem~\ref{t.main}]
By Lemma~\ref{l.ds-lower}, $d_{\mathrm{s}}\in[1,\infty)$. The Dirichlet form is nonzero. Indeed, regularity, full support of $\meas$, and $\#\ambient\ge2$, together with$\PI{\Psi}$, imply that some $u\in\domain$ has positive energy. Hence $d_{\mathrm m}\ge1$ by \cite[Proposition~2.11]{Hin10}. Thus it remains to prove $d_{\mathrm m}\le d_{\mathrm s}$.
	
Fix an integer $n\le d_{\mathrm m}$ and $s\in(d_{\mathrm s},\infty)$. The choices of $n$ and $s$ will be specified in the two cases below. By Proposition~\ref{p.finite-normalization}, there are a point $p\in \ambient$, scales $r_j\downarrow0$, rescaled Dirichlet forms $(\form^{(j)},\domain_j)$, and functions $\mathbf H_j=(H_j^1,\ldots,H_j^n) \in (C(\ambient) \cap \domain_{\loc}(\ambient))^n$ such that, on $U_j=B_j(p,R_*)$,
\begin{equation}\label{e.main-matrix}
	\Gamma^{(j)}\langle H_j^i,H_j^k\rangle=\delta_{ik}\nu_j+R_j^{ik},
\end{equation}
where
\begin{equation}\label{e.main-matrix-error}
	\varepsilon_{j}:=\sum_{i,k=1}^n\|R_j^{ik}\|_{\mathrm{TV}(U_j)}\to 0, \ \text{ and }\ \nu_j(B_j(p,1))=1.
\end{equation}
Fix $R_*=8(1+\Lambda_{\mathrm{PI}})$. By Lemma~\ref{l.uniform-cutoff}-\ref{it.cutoff1}, we may choose $\zeta_j\in\domain_j\cap C(\ambient)$ satisfying
\[0\le\zeta_j\le1,\quad\zeta_j=1\quad\text{on } B_j(p,R_*/4),\quad \zeta_j=0\quad\text{on }\ambient\setminus B_j(p,R_*/2),\]
and
\begin{equation}\label{e.main-cutoff-bound}
\sup_{j\in\bN}\left(\form^{(j)}(\zeta_j)+\|\zeta_j\|_{L^2(\ambient, \meas_{j})}^2\right)<\infty.
\end{equation}
	
 Define
\[	\kappa_j:=(\mathbf H_j)_*(\zeta_j^2\nu_j),\ 
	\sigma_j:=(\mathbf H_j)_*\bigl(\one_{B_j(p,2)}\meas_{j}\bigr)\ \text{ and }\ 
 	\lambda_j:=(\mathbf H_j)_*\left(2\Gamma^{(j)}(\zeta_j)+\zeta_j^2\meas_{j}\right).\]
By Lemma~\ref{l.euclidean-image}-\ref{it.image1} and \ref{it.image2}, after passing to a subsequence, $\kappa_j\rightharpoonup\kappa=\rho\,\mathcal L^n$ and $\lambda_j\rightharpoonup\lambda$, where $\lambda$ is a finite positive measure and
\begin{equation}\label{e.main-rho}
	\sqrt\rho\in H^1(\mathbb R^n).
\end{equation}
Moreover, since $\zeta_j=1$ on $B_j(p,1)$ and $\nu_j(B_j(p,1))=1$, the measure $\kappa$ is nonzero.
	
By Lemma~\ref{l.image-poincare}-\ref{it.image-poincare1} and \ref{it.image-poincare2}, after passing to a further subsequence, $\sigma_j\rightharpoonup\sigma$, 	where $\sigma$ is nonzero and is not a point mass, and 
\begin{equation}\label{e.main-nonpolar}
	\text{ if $n\geq2$ and $N$ is a Borel set of $\mathbb R^n$ with } \Capacity^{1,2}(N)=0,\ \text{ then }\ \sigma(N)=0.
\end{equation}
	
We record the energy estimate that will be used in both cases. For $u\in C_c^\infty(\mathbb R^n)$, the product rule, the energy-measure Cauchy--Schwarz inequality, and the chain rule, give
\begin{align}
\form_1^{(j)}(\zeta_j u(\mathbf H_j))&=\form^{(j)}(\zeta_j u(\mathbf H_j))+\|v_j\|_{L^2(\ambient,\meas_{j})}^2\\
&\leq  2\int_\ambient u(\mathbf H_j)^2 \dif \Gamma^{(j)}\langle\zeta_j\rangle + 2\int_\ambient\zeta_j^2 \dif \Gamma^{(j)}\langle u(\mathbf H_j)\rangle+\int_\ambient\zeta_j^2u(\mathbf H_j)^2\dif \meas_{j}\\
&\overset{\eqref{e.main-matrix}}{\le}2\int_{\mathbb R^n}|\nabla u|^2\dif \kappa_j +\int_{\mathbb R^n}u^2\dif \lambda_j+2\varepsilon_{j}\|\nabla u\|_{\sup}^2.\label{e.main-energy-estimate}
\end{align}
Passing to the limit, we obtain
\begin{equation}\label{e.main-energy-limit}
	\limsup_{j\to\infty}\form_1^{(j)}\bigl(\zeta_j u(\mathbf H_j)\bigr) \overset{\eqref{e.main-matrix-error},\eqref{e.main-energy-estimate}}{\le}2\int_{\mathbb R^n}|\nabla u|^2\dif \kappa+\int_{\mathbb R^n}u^2\dif \lambda,\ \text{ for all } u\in C_c^\infty(\mathbb R^n).
\end{equation}

 \begin{enumerate}[label=\textit{Case \arabic*.}, wide=0pt, topsep=5pt, parsep=0pt, itemsep=2pt]
 	
\item \textit{Suppose that $d_{\mathrm s}<2$.}
  
 Suppose, toward a contradiction, that $d_{\mathrm m}\ge2$. Choose $n=2$ and $s\in(d_{\mathrm{s}},2)$. Define $\kappa_{0,j}:=(\mathbf H_j)_*\bigl(\one_{B_j(p,1)}\nu_j\bigr)$. Then $\{\kappa_{0,j}\}_{j\in\bN}$ have mass one and uniformly compact support in $\bR^{2}$. Passing to a subsequence, there exists a measure $\kappa_{0}$ such that $\kappa_{0,j}\rightharpoonup\kappa_0$ and $\kappa_0(\mathbb R^2)=1$. Since $\zeta_j=1$ on $B_j(p,1)$, we know that 	
\begin{equation}\label{e.main-kappa0}
		0\ne\kappa_0\le\kappa=\rho\,\mathcal L^2.
\end{equation}
In particular, $\kappa_0$ has no atoms since $\kappa$ is non-atomic. By the Lebesgue differentiation theorem, there exists $a\in\supp\kappa_0$ such that
\begin{equation}\label{e.main-point2}
	\sup_{0<r\le1}r^{-2}\kappa(B(a,r))<\infty,\ \text{ and }\ \lambda(\{a\})=0.
\end{equation}
Since $a\in\supp\kappa_0$,
\begin{equation}\label{e.main-support}
	\kappa_0(B(a,r))>0
	\quad\text{for every }r\in(0,\infty).
\end{equation}
	
Choose $\eta\in C^\infty(\mathbb R)$ with $0\le\eta\le1$, $\eta(t)=1$ for $t\le0$, and $\eta(t)=0$ for $t\ge3/4$. For $\epsilon\in(0,1)$, put $r_\epsilon^-:=e^{-2/\epsilon}$ and $r_\epsilon^+:=e^{-1/\epsilon}$. Define a function $\phi_\epsilon(y):\bR^{2}\to[0,1]$ by 
\[\phi_\epsilon(y):=\eta\left(\frac{\log(|y-a|)-\log r_\epsilon^-}{\log r_\epsilon^{+}-\log r_\epsilon^-}\right),\ \text{ for }\ y\ne a,\ \text{ and }\ \phi_\epsilon(a):=1.\]
Then $\phi_\epsilon\in C_c^\infty(\mathbb R^2)$, 
\[0\le\phi_\epsilon\le1,\quad \phi_\epsilon=1\ \text{on }B(a,r_\epsilon^-),\quad \supp\phi_\epsilon \subset B(a,r_\epsilon^+),\]
and, by the chain rule, there exists $C\in(1,\infty)$ such that
 	
\begin{equation}\label{e.main-grad<}
	|\nabla\phi_\epsilon(y)|\leq \frac{\|\eta'\|_{\sup}}{|y-a|\log(r_\epsilon^+/r_\epsilon^-)} \le\frac{C\epsilon}{|y-a|},\ \text{ whenever }\ r_\epsilon^-<|y-a|<r_\epsilon^+.
\end{equation}
where the last inequality follows from $\log(r_\epsilon^+/r_\epsilon^-)\gtrsim\epsilon^{-1}$.	By a dyadic annulus decomposition,
\begin{align}
	\int_{\mathbb R^2}|\nabla\phi_\epsilon|^2\dif \kappa
	&\overset{\eqref{e.main-grad<}}{\le} C\epsilon^2\sum_{k=0}^{\lceil\log_2({r_\epsilon^+}/{r_\epsilon^-})\rceil}\frac{\kappa(B(a,2^{k+1}r_\epsilon^-))}{(2^kr_\epsilon^-)^2}\\
	&\overset{\eqref{e.main-point2}}{\le} C_{1}\epsilon^2\left(\left\lceil\log_2\frac{r_\epsilon^+}{r_\epsilon^-}\right\rceil+1\right)
	\le C_{2}\epsilon.\label{e.main-gd1}
\end{align}	
Moreover, by our choice of $a$,
\begin{equation}\label{e.main-gd2}
	\limsup_{\epsilon\downarrow0}\int_{\mathbb R^2}\phi_\epsilon^2\dif \lambda\le\limsup_{\epsilon\downarrow0}\lambda(B(a,r_\epsilon^+))=\lambda(\{a\})\overset{\eqref{e.main-point2}}{=}0.
\end{equation}
Therefore,
\begin{align}
&\phantom{\ \leq}\limsup_{\epsilon\downarrow0}\left(\limsup_{j\to\infty} \form_1^{(j)}(\zeta_j\phi_\epsilon(\mathbf H_j)) \right)\\
&\overset{\eqref{e.main-energy-limit}}{\le} \limsup_{\epsilon\downarrow0}\left(2\int_{\mathbb R^2}|\nabla\phi_\epsilon|^2\dif \kappa + \int_{\mathbb R^2}\phi_\epsilon^2\dif \lambda\right)\overset{\eqref{e.main-gd1},\eqref{e.main-gd2}}{=}0.\label{e.main-energy-zero}
\end{align}
We now obtain a contradiction from \eqref{e.main-energy-zero}. Fix $\epsilon>0$. The Portmanteau theorem and $\kappa_{0,j}\rightharpoonup\kappa_0$ imply
\[\liminf_{j\to\infty} \kappa_{0,j}(B(a,r_\epsilon^-/2))\ge\kappa_0(B(a,r_\epsilon^-/2))\overset{\eqref{e.main-support}}{>}0.\]
Recalling that $\kappa_{0,j}=(\mathbf H_j)_*(\one_{B_j(p,1)}\nu_j)$, we conclude that, for all sufficiently large $j$, there exists $x_j\in B_j(p,1)$ such that
\[|\mathbf H_j(x_j)-a|<r_\epsilon^-/2.\]
The uniform H\"older continuity of $\mathbf H_j$ implies that there is $c_\epsilon>0$, independent of $j$, such that
\[B_j(x_j,c_\epsilon)\subset B_j(p,R_*/4), \ \text{ and }\ |\mathbf H_j(y)-a|<r_\epsilon^-\quad\text{for all }y\in B_j(x_j,c_\epsilon).\]
Therefore, $\zeta_j\phi_\epsilon(\mathbf H_j)=1$ $\meas_{j}$-a.e. on $B_j(x_j,c_\epsilon)$. Since $\meas_{j}$ has full support, we thus have $\|\zeta_j\phi_\epsilon(\mathbf H_j)\|_{L^{\infty}(U_{j},\meas_{j})}=1$. Since $\zeta_j\phi_\epsilon(\mathbf H_j)\in \domain_{j,0}(U_j)$, by the Sobolev inequality \eqref{e.truncated-infty} in Lemma~\ref{l.truncated-sob}-\ref{it.sob-infty}, we have	
\[\form_1^{(j)}(\zeta_j\phi_\epsilon(\mathbf H_j))\ge C_s^{-2},\ \text{ for all sufficiently large $j$}.	\]
The lower bound is independent of $\epsilon$, contradicting \eqref{e.main-energy-zero}.
	
Hence $d_{\mathrm m}<2$. Since the martingale dimension is a positive integer or infinity, we know that 
\begin{equation}\label{e.maincase1}
d_{\mathrm m}=1	\quad\text{when } d_{\mathrm s}<2.
\end{equation}

\item \textit{Suppose that $d_{\mathrm s}\geq2$.}
	
Suppose, toward a contradiction, that $d_{\mathrm m}>d_{\mathrm s}$. Choose an integer $n$ and a real number $s$ such that $d_{\mathrm s}<s<n\le d_{\mathrm m}$.	In particular, $n\ge3$ and $s>2$. With these choices of $n$ and $s$, we apply Proposition~\ref{p.finite-normalization}, retaining the notation $\mathbf H_j$, $\kappa$, $\lambda$, and $\sigma$. Let $q=\frac{2s}{s-2}$.	The uniform Sobolev inequality \eqref{e.truncated-q} gives
\begin{equation}\label{e.main-truncated-q}
	\|v\|_{L^q(U_j,\meas_{j})}^2\le C_s\form_1^{(j)}(v),	\qquad v\in\domain_{j,0}(U_j).
\end{equation}
	
Fix $u\in C_c^\infty(\mathbb R^n)$ and set $v_j=\zeta_j u(\mathbf H_j)$. Since $\zeta_j=1$ on $B_j(p,2)$,
	\[\left(\int_{\mathbb R^n}|u|^q\dif \sigma_j\right)^{2/q}=\left(\int_{B_{j}(p,2)}|u(\mathbf{H}_{j})|^q\dif\meas_{j}\right)^{2/q}\le\|v_j\|_{L^q(U_j,\meas_{j})}^2\overset{\eqref{e.truncated-q}}{\leq}C_s\form_1^{(j)}(v_j).\]
Using $\sigma_j\rightharpoonup\sigma$ and \eqref{e.main-energy-limit}, we obtain
\begin{equation}\label{e.image-sob}
	\left(\int_{\mathbb R^n}|u|^q\dif \sigma\right)^{2/q}\overset{\eqref{e.main-energy-limit}}{\le}C\left(\int_{\mathbb R^n}|\nabla u|^2\dif \kappa+\int_{\mathbb R^n}u^2\dif \lambda\right),\ \text{for every $u\in C_c^\infty(\mathbb R^n)$}.
\end{equation}
By \eqref{e.main-rho} and Proposition~\ref{p.CapE}-\ref{it.Cap2}, there is a Borel set $N\subset\mathbb R^n$ with $\Capacity^{1,2}(N)=0$, such that, 
\[\sup_{r\in(0,1]}r^{-n}\kappa(B(a,r))<\infty,\ \text{ and }\  \lim_{r\downarrow0} r^{2-n}\lambda(B(a,r))=0,\ \text{ for every $a\in\bR^{n}\setminus N$}.\]
By \eqref{e.main-nonpolar}, $\sigma(N)=0$. Proposition~\ref{p.CapE}-\ref{it.Cap3} also gives
	\[ \liminf_{r\downarrow0}r^{-n}\sigma(B(a,r))>0,\ \text{ for $\sigma$-a.e. }a\in\bR^{n}.\]
Since $\sigma$ is nonzero, we can choose $a\in\mathbb R^n$ satisfying all three properties. In particular, there exist	$C_a,c_a,r_a>0$ such that
\begin{align}
&\kappa(B(a,r))\le C_ar^n,\ \text{ for all }\ r\in(0,1],\label{e.main-point-a1}\\
&\lambda(B(a,r))=o(r^{n-2}),\ \text{ as }\ r\downarrow0, \label{e.main-point-a2}\\
&\sigma(B(a,r))\ge c_ar^n,\ \text{ for all }\ r\in (0,r_a]. \label{e.main-point-a3}
\end{align}	

Choose $\chi\in C_c^\infty(\mathbb R^n)$ such that $0\le\chi\le1$, $\chi=1$ on $B(0,1)$, and $\supp\chi\subset B(0,2)$. For all $\epsilon\in(0,\infty)$, let
	\[u_\epsilon(y):=\epsilon^{1-n/2}\chi\left(\frac{y-a}{\epsilon}\right),\ y\in\bR^{n}.\]
By \eqref{e.main-point-a1} and \eqref{e.main-point-a2}, we know that 
\begin{equation}
\int_{\mathbb R^n} |\nabla u_\epsilon|^2\dif\kappa+\int_{\mathbb R^n}u_\epsilon^2\dif \lambda\leq C\epsilon^{-n}\kappa(B(a,2\epsilon))+C\epsilon^{2-n}\lambda(B(a,2\epsilon))\leq C_{1}
\end{equation}
for some constant $C_{1}\in(0,\infty)$ independent of $\epsilon$.

On the other hand, since $u_\epsilon=\epsilon^{1-n/2}$ on $B(a,\epsilon)$, the property \eqref{e.main-point-a3} yields
\begin{equation}
\left(\int_{\mathbb R^n}|u_\epsilon|^q\dif \sigma \right)^{2/q}\geq	\epsilon^{2-n}\sigma(B(a,\epsilon))^{2/q} \overset{\eqref{e.main-point-a3}}{\ge}	c\epsilon^{2-n+(2n/q)} =c\epsilon^{2(s-n)/s}.
\end{equation}
Since $s<n$, the exponent $2(s-n)/s<0$. Thus the left-hand side of \eqref{e.image-sob} diverges as $\epsilon\downarrow0$, a contradiction. Therefore,
\begin{equation}\label{e.maincase2}
	d_{\mathrm m}\le d_{\mathrm s}\quad\text{when }d_{\mathrm s}\ge2
\end{equation}
\end{enumerate}

Combining \eqref{e.maincase1} and \eqref{e.maincase2}, with the lower bound $d_{\mathrm{s}}\geq1$ from Lemma~\ref{l.ds-lower}, we obtain \eqref{e.main-statement}.
\end{proof}
 
\medskip

\noindent\textbf{Statement of AI use.} 
The mathematical ideas and proof strategies were developed by the authors. ChatGPT was used for proofreading, language editing, and assistance in refining certain technical details based on arguments supplied by the authors. The authors have verified the proofs and take full responsibility for their correctness.

\bigskip
\noindent\textsc{Aobo Chen}\\
Department of Mathematics, Aarhus University, 8000 Aarhus C, Denmark\\
\textit{Email:} \href{mailto:aobochen@math.au.dk}{\texttt{aobochen@math.au.dk}}

\medskip
\noindent\textsc{Mathav Murugan}\\
Department of Mathematics, University of British Columbia, Vancouver, BC V6T 1Z2, Canada\\
\textit{Email:} \href{mailto:mathav@math.ubc.ca}{\texttt{mathav@math.ubc.ca}}
\end{document}